\documentclass[12pt]{article}
\usepackage{mathrsfs}
\usepackage{amsmath}
\usepackage{amsfonts}
\usepackage{amssymb}
\usepackage{color}
\usepackage{amsmath, amsthm, amsfonts, amssymb}
\usepackage{verbatim}

\def \-{\bar}

\newtheorem{theorem}{Theorem}[section]
\newtheorem{lemma}[theorem]{Lemma}
\newtheorem{corollary}[theorem]{Corollary}
\newtheorem{proposition}[theorem]{Proposition}

\newtheorem{definition}[theorem]{Definition}
\newtheorem{example}[theorem]{Example}
\newtheorem{remark}[theorem]{Remark}

\newtheorem{conjecture}[theorem]{Conjecture}
\date{}

\begin{document}

\title{\bf Holomorphic maps from the complex unit ball to Type IV classical domains}

\author{Ming Xiao and Yuan Yuan}

\vspace{3cm} \maketitle

\begin{abstract}
We prove rigidity results for holomorphic proper maps from the complex unit ball $\mathbb{B}^n$ to the Type IV bounded symmetric domain $D^{IV}_m$ where $n \geq 4, n+1\leq m \leq 2n-3$. In addition, a classification result is established when $m=n+1.$
\end{abstract}

\bigskip

\section{Introduction}

The first part of this paper is devoted to establish new rigidity results for proper holomorphic maps from the complex unit ball to higher rank bounded symmetric domains.   The rigidity properties have been extensively studied in the past decades for proper holomorphic maps $F: \Omega_1 \rightarrow \Omega_2$, between  bounded symmetric domains $\Omega_1, \Omega_2$. The pioneer works are due to Poincar\'e \cite{P} and later to Alexander \cite{Al} when $\Omega_{1}, \Omega_{2}$ are complex unit balls. In particular,  any proper holomorphic self-map of the unit ball in $\mathbb{C}^n$ is an automorphism if $n \geq 2$ \cite{Al}. It is well-known that the rigidity properties fail dramatically for proper holomorphic maps between balls of different dimensions.
For this type of results, see \cite{HS}, \cite{Lo}, \cite{Fo1},\cite{Gl}, \cite{St}, \cite{Do}, \cite{D1} and etc. However, the rigidity properties can still be expected if certain boundary regularity of the map is assumed. See \cite{W}, \cite{Fa}, \cite{CS}, \cite{Hu1}, \cite{Hu2}, \cite{HJ1}, \cite{HJY}, \cite{Eb} and etc. The lists above are by no means to be complete.

On the other hand, it is a widely open problem to understand proper holomorphic maps $F: \Omega_{1} \rightarrow \Omega_{2}$ between bounded
symmetric domains $\Omega_{1}, \Omega_{2}$ of higher rank. The problem was first studied for the case $\Omega_1=\Omega_2$.  Any proper holomorphic self-mapping on an irreducible bounded symmetric domain of rank $\geq 2$ is an automorphism (cf. \cite{HN}, \cite{TH}).
When rank$(\Omega_1)\geq$ rank$(\Omega_2) \geq 2$ and $\Omega_{1}$ is irreducible,  it was proved by Tsai \cite{Ts} that $F$ must be a totally geodesic isometric embedding with respect to Bergman metrics. Tu \cite{Tu1} proved that the $F$ is a biholomorphism when
$\Omega_{1}$ and $\Omega_{2}$ are equal dimensional and $\Omega_{1}$ is irreducible and $\mathrm{rank}(\Omega_{1}) \geq 2.$
When rank$(\Omega_2) >$ rank$(\Omega_1)$, the studies are mainly focused on the Type I classical domains and many interesting results have been established (cf. \cite{Tu2}, \cite{Ng4}, \cite{KZ1}, \cite{KZ2} et al). Note that the total geodesy of $F$ fails in general although it is believed that $F$ should take certain special forms module automorphisms(cf. \cite{M08}, \cite{Ng4},\cite{KZ2}). In this paper, we prove new rigidity results for proper holomorphic maps from the  unit ball  in $\mathbb{C}^n$ to the  $m$-dimensional Type IV classical domain.
The isometry property still survives although the total geodesy fails in our setting. In particular, we establish a 
classification result for CR maps between their boundaries when $m=n+1.$


\medskip

 Write $\mathbb{B}^n$ for the  unit ball in $\mathbb{C}^n,$ and $D_{m}^{IV}$ for the classical type IV domain in $\mathbb{C}^m$ and equip them with the Bergman metrics $\omega_{\mathbb{B}^n}, \omega_{D^{IV}_m}$, respectively (See section 2 for explicit formulas). We say a holomorphic map $F: \mathbb{B}^n \rightarrow D^{IV}_m$ is an isometric embedding or simply an isometry if $F^{*}(\omega_{D^{IV}_m})=\lambda \omega_{\mathbb{B}^n}$ for some constant $\lambda >0.$ Write $\partial \mathbb{B}^n$ and $\partial D^{IV}_m$ for the unit sphere in $\mathbb{C}^n$ and the boundary of $D^{IV}_m,$ respectively.

\begin{theorem}\label{proper1}
Assume $n \geq 4, n+1 \leq m \leq 2n-3.$

\begin{enumerate}

\item (Local version) Let $F$ be a holomorphic map from a connected open set  $U$ in $\mathbb{C}^n$ containing $p \in \partial \mathbb{B}^n$ to $\mathbb{C}^m$. 
Assume that $F(\partial \mathbb{B}^n \cap U) \subset \partial D^{IV}_{m}$ and $F(U) \not\subset \partial D^{IV}_m.$  Then $F$ extends to a holomorphic isometric embedding from $\mathbb{B}^n$ into $D^{IV}_m$ with $F^{*}(\omega_{D^{IV}_m})=\frac{m}{n+1}\omega_{\mathbb{B}^n}.$

\item (Global version) Any algebraic proper holomorphic map $F$ from $\mathbb{B}^n$ to $D^{IV}_m$ is an isometric embedding with $F^{*}(\omega_{D^{IV}_m})=\frac{m}{n+1}\omega_{\mathbb{B}^n}.$
\end{enumerate}
\end{theorem}

\bigskip
Recall that a holomorphic function $f$ over $U \subset \mathbb{C}^n$ is called (holomorphic) Nash-algebraic, or simply algebraic if there is an irreducible holomorphic polynomial $P(z,X)$ in $X$ with coefficients polynomials of $z$ such that $P(z, f(z)) \equiv 0$ over $U.$ A holomorphic map $F$ is called algebraic if each of its components is algebraic.
In the case $m=n+1,$  the proper holomorphic maps are classified under a weaker boundary regularity condition.
\begin{proposition}\label{proper2}
Let $n \geq 4$ and $F$ be a $C^2$-smooth CR transversal CR map from an open piece of $\partial \mathbb{B}^n$ to an open smooth piece of $\partial D^{IV}_{n+1}.$ Then $F$ extends to a holomorphic isometry from $\mathbb{B}^n$ to $D^{IV}_{n+1}.$ Furthermore, $F$ is  equivalent to either
\begin{equation}\label{eqncl1}
\left(z_{1},\cdots,z_{n-1}, \frac{\frac{1}{2}\sum_{i=1}^{n-1}z_{i}^2-z_{n}^2+z_{n}}{\sqrt{2}(1-z_{n})},
\sqrt{-1}\frac{\frac{1}{2}\sum_{i=1}^{n-1}z_{i}^2+z_{n}^2-z_{n}}{\sqrt{2}(1-z_{n})} \right);
\end{equation}
or
\begin{equation}\label{eqncl2}
\left(z_1, \cdots, z_{n-1}, z_n, 1-\sqrt{1-\sum_{j=1}^n z_j^2}\right).
\end{equation}
\end{proposition}

Here two proper holomorphic maps $F_1, F_2: \Omega_1 \rightarrow \Omega_2$ are called equivalent if there exist automorphisms $\phi$ and $\psi$ of $\Omega_1$ and $\Omega_2$ respectively, such that $\psi \circ F_1 \circ \phi= F_2.$ The definition of CR transversality will be given in Section 2. 
We next list some important remarks of Theorem \ref{proper1} and Proposition \ref{proper2}.

\begin{remark}
\begin{enumerate}

\item It follows from Mok's result (\cite{M4}) that the algebraicity assumption on $F$ is necessary in Theorem \ref{proper1}. Moreover, a non-algebraic proper map from
$\mathbb{B}^n$ to $D^{IV}_{n+2}$ that is not an isometry will be constructed in
Example \ref{e025}.

\item The statement of Proposition \ref{proper2} fails if the transversality assumption is dropped. Similarly, the statement of Theorem \ref{proper1} fails if the condition $F(U) \not \subset \partial D^{IV}_m$ is dropped. See Example \ref{e0}.

\item The conclusion of Theorem \ref{proper1} fails if $m \geq 2n$(cf. Example \ref{e01}). The conclusion of Proposition \ref{proper2} fails if $n=1$ (cf. Example \ref{e10}). We suspect that Theorem \ref{proper1} holds for  all $n \geq 2, n+1 \leq m \leq 2n-1$ and Proposition \ref{proper2} holds for all $n \geq 2.$  See more details in Section 2.

\item We state the following fact as a remark of Proposition \ref{proper2}. Let $G$ be a nonconstant local $C^2$-smooth CR map from the unit sphere in $\mathbb{C}^n$ to $\partial D^{IV}_{n+1}.$ Then the image of $G$ cannot be contained in the singular set of $\partial D^{IV}_{n+1}.$ This is due to the fact that the singular set of $\partial D^{IV}_{n+1}$ is a real $(n+1)-$dimensional subvariety of $\{Z \in \mathbb{C}^{n+1}: Z\overline{Z}^t=2\}.$

\end{enumerate}

\end{remark}

\medskip

The study of proper holomorphic maps from $\mathbb{B}^n$ to $D^{IV}_m$ or CR maps from $\partial\mathbb{B}^n$ to $\partial D^{IV}_m$ is closely related to  CR maps between hyperquadrics. Rigidity properties are explored for CR maps into hyperquadrics in
\cite{BH}, \cite{BEH1, BEH2}, \cite{EHZ1, EHZ2}, \cite{ES}, \cite{Ng3}, etc (See also \cite{D1}, \cite{DL} for irrigidity phonemena in this setting.) The crucial idea to establish Theorem \ref{proper1} is to recognize $D^{IV}_m$ as
an isometric submanifold of the generalized complex unit ball so that we can apply techniques from CR geometry. As an important step to approach Theorem \ref{proper1}, we establish a rigidity result first for holomorphic proper maps from the unit ball to the generalized balls. This part is motivated by the framework of Baouendi-Ebenfelt-Huang (\cite{BEH2}) and is proved by a very similar argument.

\medskip

In another direction, since the work of Bochner \cite{B} and Calabi \cite{C}, lots of efforts have been made to understand the local holomorphic isometry $F: U \rightarrow \Omega_2$ with respect to the (normalized) Bergman metrics of $\Omega_1$ and $\Omega_2$ respectively, i.e. $F^*\omega_{\Omega_2} = \lambda \omega_{\Omega_1}$ on $U$, where $U \subset \Omega_1$ is a connected open set. This problem is largely motivated by the algebraic dynamics problem considered by Clozel-Ullmo \cite{CU} (cf. \cite{M4} \cite{MN} \cite{DiL} \cite{Yu} \cite{FHX} et al for further developments).
 Mok (\cite{M4}) proved  that $F$ extends to an algebraic proper holomorphic isometry from $\Omega_1$ to $\Omega_2$ . Assume dim$_{\mathbb{C}}(\Omega_1) \geq 2$ and $\Omega_1$ is irreducible. Mok proved that $F$ is totally geodesic if rank$(\Omega_1) \geq 2$ \cite{M4}. When rank$(\Omega_1)=1$, 
and $\Omega_2$ is the product of complex unit balls, $F$ is also totally geodesic by works of Mok \cite{M2}, Ng \cite{Ng2}, the second author and Zhang \cite{YZ}. However, when $\Omega_1=\mathbb{B}^n$ and $\Omega_2$ is a bounded symmetric domain other than the product of unit balls, the total geodesy fails dramatically \cite{M5}. In fact, assuming that $\Omega_2$ is irreducible and rank$(\Omega_2) \geq 2$, Mok constructed a non-totally geodesic holomorphic isometry from $\mathbb{B}^n$ into $\Omega$ by using the theory of variety of minimal rational tangents \cite{M5}. In the next theorem, we classify the local holomorphic isometries from $\mathbb{B}^n$ into $D^{IV}_{n+1}$. We refer to \cite{XY} for the study in the general case and note that this problem is studied independently by Chan-Mok in \cite{CM}.

\begin{theorem}\label{T1}
Let $n \geq 2$ and $F$ be any holomorphic isometry from an open set $U \subset \mathbb{B}^n$ into $D^{IV}_{n+1}$ satisfying
\begin{equation}\notag
F^*\omega_{D^{IV}_{n+1}} = \lambda \omega_{\mathbb{B}^n}~\text{on}~U \end{equation}
for some $\lambda >0.$  Then $F$ is equivalent to either  the map in (\ref{eqncl1}) or the map in (\ref{eqncl2}).
\end{theorem}

The novelty in Theorem \ref{proper1} is that, as long as the codimension is small, the isometry is implied by properness, which is the converse statement of Mok's theorem (\cite{M4}). This can be applied to obtain the following corollary.

\begin{corollary}\label{coro14}
Assume $n \geq 4, K \geq 1, N_1, \cdots, N_K \leq 2n-3$. Let $F=(F_1, \cdots, F_K)$ be a holomorphic map from a connected open set
$U \subset \mathbb{B}^n$ into the product of Type IV domains $\Omega = D^{IV}_{N_1} \times \cdots \times D^{IV}_{N_K}$ satisfying the following isometric equation:
\begin{equation}\label{eeq}
 \lambda \omega_{\mathbb{B}^n} = \sum_{l=1}^K F^*_l \omega_{D^{IV}_{N_l}} = F^*\omega_\Omega ~\text{on}~U
\end{equation}
 for some positive constant $\lambda$. Then each $F_{l}, 1 \leq l \leq K$, is either a constant map or  extends to a holomorphic isometric embedding from $\mathbb{B}^n$ to $D^{IV}_{N_l}$ with $F^{*}(\omega_{D^{IV}_{N_l}})=\frac{N_l}{n+1} \omega_{\mathbb{B}^n}.$
\end{corollary}

We would like to point out that the same conclusion can be made with the slightly more general assumption
\begin{equation}\notag
  \omega_{\mathbb{B}^n} = \sum_{l=1}^m \lambda_{l} F^*_l \omega_{D^{IV}_{N_l}}~\text{on}~U
\end{equation}
instead of (\ref{eeq}) for positive constants $\lambda_l.$ The difference in the proof is that the algebraicity in this case follows from the argument in \cite{HY1, HY2}.

\medskip

\bigskip

{\bf Acknowledgement}: The authors are grateful to Professors J. D'Angelo, X. Huang, N. Mok and S. Ng for helpful discussions. The second author is supported in
part by National Science Foundation grant DMS-1412384 and the seed grant program at Syracuse University. This work is also supported by the National Science Foundation under Grant No. 0932078 000 while the second author was in residence at the Mathematical Sciences Research Institute in Berkeley, California, during the 2016 Spring semester.

\section{Preliminaries and Some Examples}
An irreducible Hermitian symmetric manifold of non-compact type can be realized as the four types of Cartan's classical domains  and two exceptional cases in complex Euclidean spaces(cf. \cite{H2} \cite{M1}). In particular, the complex unit ball $\mathbb{B}^n = \{z=(z_1, \cdots, z_n) \in \mathbb{C}^n: |z|^2 <1\}$ in $\mathbb{C}^n$ is a special case of the type I classical domain. The Bergman kernel is given by $$K_{\mathbb{B}^n}(z, \bar z)= c_{I} \left( 1-|z|^2 \right)^{-(n+1)}.$$ 
The type IV classical domain is defined as $$D^{IV}_m = \{Z =(z_1, \cdots, z_m) \in \mathbb{C}^m | Z \overline{Z}^t <2 ~\text{and} ~ 1- Z \overline{Z}^t + \frac{1}{4} |Z Z^t|^2 >0 \},$$
where $Z^t$ is the transpose of $Z$
and the Bergman kernel function $K_{D^{IV}_m}(Z, \bar Z)$ is explicitly given by
\begin{equation}
K_{D^{IV}_m}(Z, \bar Z) = c_{IV} \left( 1 -  Z\overline{Z}^t +\frac{1}{4} |Z Z^t|^2 \right)^{-m}. 
\end{equation}
Here $c_I, c_{IV}$ are positive constants depending on the dimensions. 
(cf. \cite{H2} \cite{M1}). The Bergman metric  
$$\omega_{\Omega}(Z) :=\sqrt{-1} \partial \bar\partial \log K_{\Omega}(Z, \bar Z)$$
of a bounded symmetric domain $\Omega$ is K\"ahler-Einstein as the Bergman kernel function is invariant under the holomorphic automorphisms. A straightforward computation shows that $\partial D^{IV}_m = \{Z \in \mathbb{C}^m: 1- Z\overline{Z}^t +\frac{1}{4} |Z Z^t|^2 =0, Z\overline{Z}^t \leq 2 \}.$
Moreover, if $m \geq 2$, the singular set $P$ of $\partial D^{IV}_m$ is given by
\begin{equation}\notag
\begin{split}
P = \{Z \in \mathbb{C}^m | 1- Z\overline{Z}^t +\frac{1}{4} |Z Z^t|^2 =0, Z\overline{Z}^t= 2 \}, \\
\end{split}
\end{equation}
which is a real $m$-dimensional variety. 

\medskip

In \cite{M5}, Mok constructed non-totally geodesics holomorphic isometries from the complex unit ball into the irreducible bounded symmetric domain $\Omega$ when rank$(\Omega) \geq 2$ by using the theory of varieties of minimal rational tangents. For the purpose of the current paper, we formulate Mok's theorem merely in the case of Type IV domains. 

\begin{theorem}[Mok] \label{mok}
Assume $m \geq 2$.
\begin{itemize}
\item[(i)] If   $F: \mathbb{B}^n \rightarrow D^{IV}_m$ is a holomorphic isometry, 
then $n \leq m-1.$
\item[(ii)] There exists a non-totally geodesic holomorphic isometric embedding $G: \mathbb{B}^{m-1}
 \rightarrow D^{IV}_m$ with $G^* \omega_{D^{IV}_m} = \lambda \omega_{\mathbb{B}^{m-1}}$ for some $\lambda > 0.$
\end{itemize}
\end{theorem}

We now make some remarks on our main theorems. The following simple fact explains why we only consider $m \geq n+1$ in our main theorems.
\begin{lemma}\label{l2}
Let $n \geq 2
$ and $F: \mathbb{B}^n \rightarrow D^{IV}_m$ be a proper holomorphic map. Then $m \geq n+1$.
\end{lemma}

\begin{proof}
It suffices to show that there is no proper holomorphic map for $m= n \geq 2$. Suppose that $F: \mathbb{B}^n \rightarrow D^{IV}_n$ is a proper holomorphic map. Note that $\mathbb{B}^n$ and $D^{IV}_{n}$ are bounded complete circular domains. Then $F$ is algebraic by Bell's algebraicity result \cite{B}. It then extends holomorphically to a neighborhood $U$ of an open piece of the boundary $\partial \mathbb{B}^n$.  Note that the image of $\partial \mathbb{B}^n \cap U$ cannot be contained in the singular set of $\partial D^{IV}_n.$
One then easily achieves a contradiction  since $\partial D^{IV}_n$ is Levi-degenerate at any smooth point.
\end{proof}

\begin{remark}
The statement of Lemma \ref{l2} fails if $n=1.$ $F(z)=\sqrt{2}z$ is indeed a holomorphic isometry from the unit disc $\Delta$ to $D^{IV}_1.$
\end{remark}

As mentioned in Section 1, the assumption that $F$ is algebraic in Theorem \ref{proper1} is a necessary condition by Mok's result. We now give an example of non-algebraic proper map from
$\mathbb{B}^n$ to $D^{IV}_{n+2}$ that is not an isometry. In the following context,
we define $\sqrt{w}=\sqrt{r} e^{\sqrt{-1}\frac{\theta}{2}},$ for any complex number $w=re^{\sqrt{-1}\theta}$ with $\theta \in (-\pi, \pi].$

\begin{example}\label{e025}
Assume $n \geq 2$ and let $H=(h_1, \cdots, h_{n+1}): \mathbb{B}^n \rightarrow \mathbb{B}^{n+1}$ be a proper holomorphic map with $H(0)=0$ such that $H$ is continuous up to the boundary $\partial \mathbb{B}^n$ but is not twice continuously differentiable up to any open piece of boundary $\partial \mathbb{B}^n$(See \cite{Do}). In particular, $H$ is not algebraic. Define
$$g=1-\sqrt{1-\sum_{j=1}^{n+1}h_j^2}.$$ Then $g$ is holomorphic in $\mathbb{B}^n$ and satisfies $$g= \frac{1}{2} \left(\sum_{j=1}^{n+1}h_j^2 + g^2 \right).$$ It follows that
$$\sum_{j=1}^{n+1}|h_j|^2 + |g|^2 - \frac{1}{4} \left|\sum_{j=1}^{n+1}h_j^2+g^2\right|^2 = \sum_{j=1}^{n+1}|h_j|^2, $$
and the straightforward computation verifies that $\sum_{j=1}^{n+1}|h_j|^2 + |g|^2 < 2$ in $\mathbb{B}^n.$
Thus $G=(H, g): \mathbb{B}^n \rightarrow D^{IV}_{n+2}$ is a proper holomorphic map and is not an isometry.
\end{example}

We now recall the definition of CR transversality. Let $M \subset \mathbb{C}^n, M' \subset \mathbb{C}^N$ be two CR submanifolds. A CR map $H: M \rightarrow M'$ is called CR transversal at $p \in M$ if $dH(\mathbb CT_pM)$ is not contained in $\mathcal V'_{F(p)} \oplus \overline{\mathcal V'_{F(p)}},$
where $\mathcal V'$ is the CR bundle of $M'.$  Example \ref{e0} shows that the statement of Proposition \ref{proper2} fails if the transversality assumption is dropped, and the statement of Theorem \ref{proper1} fails if the assumption $F(U) \not \subset \partial D^{IV}_m$ is dropped.

\begin{example}\label{e0}
Let $n \geq 1, m \geq 3$. Let $G$ be a holomorphic map from a small neighborhood $U$ of $p \in \partial \mathbb{B}^n$ defined by
$$G=\left(\frac{1+\psi}{\sqrt{2}}, \frac{1-\psi}{\sqrt{-2}}, 0, ...,0\right),$$
where there are $m-2$ zero components and $\psi$ is any holomorphic function in $U.$ It is easy to see that $G$ maps $U$ to $\partial D^{IV}_{m}$ and does not extend to an isometry.
\end{example}

Example \ref{e01} show that the statement of Theorem \ref{proper1} fails when $m \geq 2n.$ Example \ref{e10} shows that the statement of Proposition \ref{proper2} fails if $n=1$  and  the statement of
Theorem \ref{proper1} fails for $n=1, m \geq 2.$

\begin{example}\label{e01}
Assume $n\geq 2$ and let $H=(h_1, \cdots, h_{2n-1})=(z_1, \cdots, z_{n-1}, z_1 z_n, \cdots, z_n^2): \mathbb{B}^n \rightarrow \mathbb{B}^{2n-1}$ be the well-known Whitney map. Define the holomorphic function
$$g=1-\sqrt{1-\sum_{j=1}^{2n-1}h_j^2}$$
on $\mathbb{B}^n.$  Note that
$$\sum_{i=1}^{2n-1}|h_{i}|^2 +|g|^2-\frac{1}{4}\left|\sum_{i=1}^{2n-1} h_{i}^2 + g^2\right|^2=\sum_{i=1}^{n-1}|z_{i}|^2 +|z_{n}|^2 \sum_{i=1}^{n}|z_{i}|^2. $$
It is easy to verify that $G=(H,g)$ is a holomorphic polynomial  proper map from $\mathbb{B}^n$ to $D^{IV}_{2n},$ while it is not an isometry. 
\end{example}

\begin{example}\label{e10}
Assume  $k \geq 1$.
Let $G_{k}$  be the holomorphic map from the unit disc  $\Delta \subset \mathbb{C}$ to $D^{IV}_2$ defined by
\begin{equation}\label{exx}
G_k=(z^k, 1-\sqrt{1-z^{2k}}).
\end{equation} It is easy to verify that $G_{k}$ is a proper
holomorphic map.  Moreover, $G_{k}$ is an isometry from $\Delta$ to $D^{IV}_2$ if and only if $k=1$ (See Proposition \ref{con}).
\end{example}

Note that all the proper map examples above  are constructed from proper maps between balls. It would be reasonable to ask whether all holomorphic proper maps are obtained in this manner. More precisely, let $H=(h_{1},..., h_{m}): \mathbb{B}^n \rightarrow D^{IV}_m$ be a proper holomorphic map. Does there always exist a holomorphic proper map $F=(f_{1},...,f_{m'}): \mathbb{B}^n \rightarrow \mathbb{B}^{m'}$ with $m' \leq m,$ such that
the following relation holds:
\begin{equation}\label{eqnconj00}
\sum_{i=1}^m |h_{i}|^2- \frac{1}{4}\left| \sum_{i=1}^m h_{i}^2 \right|=\sum_{j=1}^{m'} |f_{j}|^2~~?
\end{equation}
The answer is negative in general (See Example \ref{exhp0}), while we suspect that it is true when $m$ is small
compared to $n.$
\begin{example}\label{exhp0}
Assume $n \geq 2$.
Let $H=(h_{1},...,h_{4n-1})$ be the proper holomorphic map from $\mathbb{B}^n$ to $D^{IV}_{4n-1}$
defined by
$$\left(z_{n},\frac{z_{1}}{\sqrt{2}}, \frac{z_{1}}{\sqrt{-2}},..., \frac{z_{n-1}}{\sqrt{2}}, \frac{z_{n-1}}{\sqrt{-2}}, \frac{z_{1}z_{n}^2}{2\sqrt{2}},\frac{z_{1}z_{n}^2}{2\sqrt{-2}}..., \frac{z_{n-1}z_{n}^2}{2\sqrt{2}}, \frac{z_{n-1}z_{n}^2}{2 \sqrt{-2}}, \frac{z_{n}^3}{2\sqrt{2}}, \frac{z_{n}^3}{2\sqrt{-2}}\right).$$
By computation, we have
\begin{equation}\label{eqn4n}
\sum_{i=1}^m |h_{i}|^2- \frac{1}{4}\left| \sum_{i=1}^m h_{i}^2 \right|=\sum_{i=1}^{n}|z_{i}|^2-\frac{1}{4}\left(|z_{n}|^4(1-\sum_{i=1}^n |z_{i}|^2) \right).
\end{equation}
Note that the right hand side of (\ref{eqn4n}) cannot be written as sum of norm squares of holomorphic functions (cf. the proof of Proposition \ref{con}).
This implies there is no proper holomorphic map $F: \mathbb{B}^n \rightarrow \mathbb{B}^{m'}$ such that (\ref{eqnconj00}) holds.

\end{example}

On the other hand, it would still be interesting to study how the initial  boundary regularity of the map $F$
can be relaxed in the hypotheses of Proposition \ref{proper2}. 
We make the following conjecture along this line:

\begin{conjecture}\label{conjecture1}
Let $n \geq 2$. Any proper holomorphic map from $\mathbb{B}^n$ to $D^{IV}_{n+1}$ is a holomorphic isometry (and thus it is equivalent to one of the maps in  Theorem \ref{T1}).
\end{conjecture}

\begin{remark}
Example \ref{e10} also shows that the statement of Conjecture \ref{conjecture1} fails when $n=1.$
\end{remark}

We end this section with the following observation on isometric constants. Note that this result is obtained independently by Chan-Mok \cite{CM}.

\begin{proposition}\label{con}

Let $F: \mathbb{B}^n \rightarrow D^{IV}_m$ be a local holomorphic isometric embedding with respect to the Bergman metrics of isometric constant $\lambda >0$, i.e.
\begin{equation}\label{ccc}
F^*\omega_{D^{IV}_m} = \lambda \omega_{\mathbb{B}^n}.
\end{equation}
Then
\begin{itemize}
\item If $n \geq 2,$ then $\lambda =m/(n+1).$
\item If $n=1,$ then $\lambda=m/2$ or $m.$
\end{itemize}
\end{proposition}

\begin{proof}
Without loss of generality, we may assume $F(0)=0$ by composing the automorphisms of $\mathbb{B}^n$ and $D^{IV}_m$. By the standard reduction, (\ref{ccc}) is equivalent to
\begin{equation}\label{cc}
\left(1-|z|^2 \right)^{\lambda(n+1)/m} = 1- F(z)\overline{F(z)}^t + \frac{1}{4} |F(z)F(z)^t|^2
\end{equation}
that holds in a neighborhood $U$ of $0.$ Note that the signature of the left hand side of (\ref{cc}) is either $(1, s)$ or $(2, s)$ for some integer $s \geq 0$, meaning that it can be written as linear combination of $1$ or $2$ sum of squares minus $s$ sum of square of linearly independent holomorphic functions over positive real numbers. It fact, it is of signature $(1, s)$ if and only if $F(z)F(z)^t \equiv 0$. Obviously, the left hand side is of finite rank if and only if $\lambda(n+1)/m \in \mathbb{N}$ (cf. \cite{Um}). Write $p=\lambda(n+1)/m$. We first consider the case $n=2.$ Assume $p \geq 2$. Applying binomial formula,
$$\left( 1 - \sum_{j=1}^n |z_j|^2 \right)^p = \sum_{k=0}^p (-1)^k \binom{u}{k} \left(\sum_{j=1}^n |z_j|^2\right)^{p-k}.$$
Note that the monomials on the right hand side are linearly independent and thus the right hand side is of signature $(r', s')$ with $r' \geq 3$ if $n \geq 2$. Therefore, the left hand side and right hand side of (\ref{cc}) have different signatures and this is a contradicttion(cf. \cite{D2} or \cite{Um}). This implies the first part of the Lemma for $n \geq 2.$ In the case $n=1,$ one can similarly get a contradiction if $p \geq 3.$ The Lemma is then established.
\end{proof}

\begin{remark}
Both $\lambda=m/2$ and $m$ can be obtained when $n=1$ in Proposition \ref{con}. Indeed, the map $G_1$ in (\ref{exx}) is a holomorphic isometry with isometric constant $\lambda=1$ while the isometric map $F(z)=\left( \sqrt{2}z, 0\right): \Delta \rightarrow D^{IV}_2$ has the isometric constant $\lambda=2.$
\end{remark}


\begin{remark}
The same argument yields that if $F: \mathbb{B}^n \rightarrow \Omega$ is a holomorphic isometric embedding from a unit ball into a classical symmetric domain with respect to Bergman metric of isometric constant $\lambda$, then $\lambda (n+1) / (p+q)$, $\lambda (n+1) / (m-1)$ or $\lambda (n+1) / (m+1)$ is a positive integer when $\Omega = D^I_{p, q}, D^{II}_m$ or $D^{III}_m$ respectively.
\end{remark}



\section{On CR maps from $\partial\mathbb{B}^n$ to $\partial D^{IV}_m$}

We will prove Theorem \ref{proper1}, Proposition \ref{proper2} and Corollary \ref{coro14} in this section. To this end we first introduce the generalized complex balls and their indefinite metrics.

\begin{definition}
Let $0 \leq l \leq n, n \geq 2$.
$$\mathbb{B}^{n+l}_l=\{(w_1, \cdots, w_l, z_1, \cdots, z_n) \in \mathbb{C}^{n+l} | \sum_{j=1}^n |z_j|^2 - \sum_{j=1}^l |w_j|^2 <1 \}$$ is called the generalized complex ball with signature $l$ in $\mathbb{C}^{n+l}$ and $$\omega_{\mathbb{B}^{n+l}_l} = -\sqrt{-1}\partial\bar\partial \log \left( 1+\sum_{j=1}^l |w_j|^2 - \sum_{j=1}^n |z_j|^2 \right)$$ is the indefinite K\"ahler metric on $\mathbb{B}^{n+l}_l$.
\end{definition}

Note that $\mathbb{B}^{n+l}_l$ is the  indefinite complex hyperbolic space form and $\mathbb{B}^n_0$ is the complex unit ball $\mathbb{B}^n$ and $\omega_{\mathbb{B}^n_0}$ is the Bergman metric on $\mathbb{B}^n$ up to a positive constant $n+1$. The holomorphic automorphism group of $\mathbb{B}^{n+l}_l$ can be described as follows. We first embed $\mathbb{B}^{n+l}_l$ into $\mathbb{P}^{n+l}$ as an open set by standard coordinate embedding: $$(w_1, \cdots, w_l, z_1, \cdots, z_n) \rightarrow [1, w_1, \cdots, w_l, z_1, \cdots, z_n].$$
Therefore $\mathbb{B}^{n+l}_l=\{ [s, w_1, \cdots, w_l, z_1, \cdots, z_n] \in \mathbb{P}^{n+l} | s\not= 0,  |s|^2+\sum_{j=1}^l |w_j|^2 - \sum_{j=1}^n |z_j|^2 > 0\}$ in homogenous coordinates. The holomorphic automorphism group of $\mathbb{B}^{n+l}_l$ is a subgroup of the holomorphic automorphism group of $\mathbb{P}^{n+l}$ given by
$$U(n+l+1, l+1) = \{A \in GL(n+l+1, \mathbb{C}) | A E(l+1, n+l+1) \bar A^t = E(l+1, n+l+1)\},$$
where $E(l+1, n+l+1)$ is a diagonal matrix with $-1$ on the first $l+1$ diagonal elements and then $1$ on the next $n$ diagonal elements. It is easy to check that $\omega_{\mathbb{B}^{n+l}_l}$ is invariant under $U(n+l+1, l+1)$.

The connection between classical domain of Type IV and the generalized complex balls is as follows. Define $L: D^{IV}_m \rightarrow \mathbb{B}^{m+1}_1$ by $L(z_1, \cdots, z_m) = [1, \frac{1}{2} \sum_{j=1}^m z_j^2, z_1, \cdots, z_m]$. One can easily check that $L$ is a proper holomorphic map and furthermore is a holomorphic isometric embedding up to a positive constant $1/m$, i.e. $L^* \omega_{\mathbb{B}^{m+1}_1} = \frac{1}{m} \omega_{D^{IV}_m}$. The following is a key result in proving Theorem \ref{proper1}.

\begin{theorem}\label{propergene1}
Assume $n \geq 4, n+1 \leq N \leq 2n-2$ and assume $U$ is a connected open set in $\mathbb{C}^n$ containing $p \in \partial \mathbb{B}^n$, $V$ is an open set in $\mathbb{C}^N$ containing $q \in \partial \mathbb{B}^N_1$. Let $F$ be a holomorphic map on $U$ such that
\begin{itemize}
\item $F(U) \not\subset \partial \mathbb{B}^N_1$,
\item $F(U \cap \partial \mathbb{B}^n) \subset \partial \mathbb{B}^N_1$.
\end{itemize}
Then the following statement holds:
\begin{itemize}
\item If $N=n+1,$  then $F$ is equivalent to
\begin{equation}\label{eqn}
(0, z_1, \cdots, z_n)
\end{equation}
\item If $N > n+1,$ then $F$ is equivalent to
\begin{equation}\label{eqn1}
(\psi, \psi, z_1, \cdots, z_n, \bf 0),
\end{equation}
where ${\bf 0}$ denotes the $(N-n-2)-$dimensional zero row vector and $\psi$ is some holomorphic function on $U$.
\end{itemize}
\end{theorem}

The proof follows from a very similar argument as in \cite{BEH2}. The crucial techniques in the proof
include the normal form type argument, Huang's Lemma \cite{Hu1} and its generalizations \cite{BEH2}, moving point trick \cite{Hu1}, and a transverality result in \cite{BER2}. To make it easier for the readers, we will sketch a proof in Appendix.
%
The following corollary follows from Theorem \ref{propergene1}. 

\begin{corollary}\label{propergene2}
Assume the same assumptions as in Theorem \ref{propergene1}. Then $F$ is a local holomorphic isometric embeddings from $\mathbb{B}^n$ to $\mathbb{B}^{N}_1$, i.e. $F^* \omega_{\mathbb{B}^{N}_1} =  \omega_{\mathbb{B}^{n}} / (n+1)$ on $U \cap \mathbb{B}^n.$
\end{corollary}

\begin{proof}
It is straightforward to check that both maps in (\ref{eqn}) and (\ref{eqn1}) are isometries from $\mathbb{B}^n$ to $\mathbb{B}^N_{1}$, i.e. $F^* \omega_{\mathbb{B}^{N}_1} =  \omega_{\mathbb{B}^{n}} / (n+1)$. Then Theorem \ref{propergene2} follows directly from Theorem \ref{propergene1} as $\omega_{\mathbb{B}^{n}}$ and $\omega_{\mathbb{B}^{N}_1}$ are invariant under
automorphisms..
\end{proof}


We are now at the position to prove Theorem \ref{proper1}, Proposition \ref{proper2} and Corollary \ref{coro14}.  Theorem \ref{T1} will be proved in the next section.

\medskip

{\bf Proof of Theorem \ref{proper1}.} The composition $L \circ F$ is a holomorphic map satisfying the assumption in Corollary \ref{propergene2}. Thus $L \circ F$ is a local holomorphic isometry from $\mathbb{B}^n$ to $\mathbb{B}^{m+1}_1$ by Corollary \ref{propergene2}, i.e. $(L \circ F)^*\omega_{\mathbb{B}^{m+1}_1} =\frac{1}{n+1}  \omega_{\mathbb{B}^n}$ on $U \cap \mathbb{B}^n$. Note that $$ \omega_{\mathbb{B}^n} / (n+1)= (L \circ F)^*\omega_{\mathbb{B}^{m+1}_1} =  F^* \left(L^* \omega_{\mathbb{B}^{m+1}_1}\right) = \frac{1}{m} F^*\omega_{D^{IV}_{m}}~\text{on}~U \cap \mathbb{B}^n.$$
Therefore $F$ extends to a global holomorphic isometry from $\mathbb{B}^n$ to $D^{IV}_m$ by Mok's theorem \cite{M4}.
\qed

\medskip

{\bf Proof of Proposition \ref{proper2} assuming Theorem \ref{T1}.} An easy computation verifies that  $\partial D^{IV}_{n+1}$ is holomorphically nondegenerate (See \cite{BER1}) at any smooth point. Indeed, $\partial D^{IV}_{n+1}$ is  $2$-nondegenerate in the sense of Baouendi-Huang-Rothschild \cite{BHR} at any smooth point. It follows from a regularity result in [KLX] that $F$ is real analytic
in some open piece $V$ of $\partial \mathbb{B}^n.$   $F$ thus extends holomorphically to a neighborhood $U$ of $V$ in $\mathbb{C}^n.$  $F(U)$ cannot be contained in $\partial D^{IV}_{n+1}$ by the CR transversality of $F.$ We conclude that $F$ is an isometry by Theorem \ref{proper1}. Then Proposition \ref{proper2} follows from Theorem \ref{T1}.

\medskip

{\bf Proof of Corollary \ref{coro14}.} First we note that by the algebraic extension theorem of Mok \cite{M4} (following from the argument in \cite{HY1, HY2} in this case as well), $F$ is an algebraic map. By holomorphic continuation, $F$ can be extended along some path to an open set $V$ in $\mathbb{C}^n$ containing a generic boundary point of $\mathbb{B}^n$ and the holomorphic isometry equation (\ref{eeq}) is preserved along the path. Therefore, there must exist a nonconstant holomorphic map $F_l$ such that $F(V\cap \partial\mathbb{B}^n) \subset \partial D^{IV}_{N_l}$ and $F(V) \not\subset \partial D^{IV}_{N_l}$. It follows from Theorem \ref{proper1} that $F_l$ is a holomorphic isometric embedding. The corollary then follows from an induction argument.

\section{On holomorphic isometries from $\mathbb{B}^n$ into $D^{IV}_{n+1}$}

\subsection{Examples of holomorphic isometries}

Write $z=(z_{1},...,z_{n})$ to be the coordinates in $\mathbb{C}^n$ for $n \geq 2.$ Let $R^{IV}_n: \mathbb{B}^n\rightarrow D^{IV}_{n+1}$ be defined as
\begin{equation}\label{stad}
R^{IV}_n=(f_{1},...,f_{n-1},f_{n},f_{n+1}),
\end{equation}
where $f_{i}=z_{i}, 1 \leq i \leq n-1, f_{n}=\frac{P_{n}}{Q}, f_{n+1}=\frac{P_{n+1}}{Q},$
$$P_{n}=\frac{1}{2}\sum_{i=1}^{n-1}z_{i}^2-z_{n}^2+z_{n}, P_{n+1}=-\sqrt{-1}\left(\frac{1}{2}\sum_{i=1}^{n-1}z_{i}^2+z_{n}^2-z_{n}\right), Q=\sqrt{2}(1-z_{n}).$$
It is easy to verify that
$$\sum_{i=1}^{n}f_{i}^2=\frac{\sum_{i=1}^{n-1}z_{i}^2}{1-z_{n}}, ~~|f_{n}|^2+|f_{n+1}|^2=\frac{1}{4}\frac{|\sum_{i}^{n-1} z_{i}^2|^2}{|1-z_{n}|^2}+|z_{n}|^2.$$
It then follows that $R^{IV}_n$ is a holomorphic isometry from $\mathbb{B}^{n}$ to $D^{IV}_{n+1}.$ In fact, we will show that $R^{IV}_n$ is the unique rational holomorphic isometry from $\mathbb{B}^{n}$ to $D^{IV}_{n+1}$ up to holomorphic automorphisms.

\medskip

For any $\theta \in [0, \pi/4)$, define $h_\theta(z) = 1+ 2 \sqrt{-1} \sin(2\theta) z_n - z_n^2 -\cos(2\theta) \left(\sum_{j=1}^{n-1} z_j^2\right)$. Let $I_{n, \theta} = (f_1, \cdots, f_{n+1})$ be \begin{equation}\label{litchi11}
\begin{split}
f_1(z) &= z_1, \cdots, f_{n-1}(z) = z_{n-1};\\
f_{n}(z) &= \frac{(\cos\theta +\sqrt{-1} \sin\theta z_n) - \cos\theta \sqrt{h_\theta(z)}}{\cos(2\theta)};\\
f_{n+1}(z) &= \frac{(-\sqrt{-1}\sin\theta + \cos\theta z_n) + \sqrt{-1}\sin\theta \sqrt{h_\theta(z)}}{\cos(2\theta)}.
\end{split}
\end{equation}
Then $I_{n, \theta}$ is  a holomorphic isometry from $\mathbb{B}^{n}$ to $D^{IV}_{n+1}.$ In particular, 
\begin{equation}\label{eqnstad}
I_{n, 0}=\left(z_1, \cdots, z_{n-1}, 1-\sqrt{1-\sum_{j=1}^n z_j^2}, z_n\right).
\end{equation}
 We will also show that $I_{n, 0}$ is the unique irrational holomorphic isometry from $\mathbb{B}^{n}$ to $D^{IV}_{n+1}$ up to holomorphic automorphisms.

\subsection{Singularities of holomorphic isometries}
The rational holomorphic isometry $R^{IV}_n$ given in previous sections is not a totally geodesic and only produces singularities at one single point on the boundary $\partial\mathbb{B}^n$. When $n \geq 2$, one can easily avoid passing through this point by slicing $\mathbb{B}^n$ with a complex linear hyperplane. Therefore, one obtains a holomorphic polynomial isometry from $\mathbb{B}^{n-1}$ into $D^{IV}_{n+1}$. In particular, this answers the question raised by Mok in \cite{M3} (Question 5.2.2) in the negative for Type IV domains  and counter-examples for other types of classical domains are given in \cite{XY}. Note that this type of examples are discovered independently by Chan-Mok \cite{CM}. 

\begin{theorem}\label{tian}
Assume $m \geq n+2$.
There exist non-totally geodesic
 holomorphic isometries from the unit ball $\mathbb{B}^{n}$ into $D^{IV}_m$ that can be extended holomorphically to $\mathbb{C}^{n}$.
\end{theorem}

\begin{proof}
The map can be indeed chosen to be polynomial. The polynomial holomorphic isometries is given by
\begin{equation}\notag
  \left(z_1, \cdots, z_{n}, \frac{\sqrt{2}}{4}\sum_{i=1}^{n}z_{i}^2, \frac{\sqrt{-2}}{4}\sum_{i=1}^{n}z_{i}^2, 0, \cdots, 0\right) : \mathbb{B}^{n} \rightarrow D^{IV}_{m}  ~{\rm for}~ m \geq n+2, n\geq 1,
\end{equation}
where there are $(m-n-2)$ zero components in the map.
\end{proof}

\subsection{Proof of Theorem \ref{T1}}

This subsection is devoted to establish Theorem  \ref{T1}. For that we will first describe the holomorphic automorphism group action on $D^{IV}_m$ in terms of the Borel embedding (cf. \cite{H1} \cite{M1}).
The hyperquadric $\mathbb{Q}^m $, the compact dual of $D^{IV}_n$ is defined by $\mathbb{Q}^m := \{[z_1, \cdots, z_{m+2}] \in \mathbb{P}^{m+1} | z_1^2+ \cdots + z_m^2 = z^2_{m+1} + z_{m+2}^2 \}$.
The Borel embedding $D^{IV}_m \subset \mathbb{Q}^m \subset \mathbb{P}^{m+1}$ is given by $$Z= (z_1, \cdots, z_m) \rightarrow \left[z_1, \cdots, z_m, \frac{1+ \frac{1}{2}Z Z^t}{\sqrt{2}}, \frac{1- \frac{1}{2}Z Z^t}{\sqrt{-2}}\right].$$
The holomorphic automorphism group of $D^{IV}_m$ is given by 
\begin{equation}\notag
{\rm Aut}(D^{IV}_{m}) = \left\{
\begin{bmatrix}
A & B\\
C & D\\
\end{bmatrix}  \in O(m, 2, \mathbb{R}) |  {\rm det}(D)>0,
\right\}
\end{equation}
where $A\in M(m, m, \mathbb{R}), B \in M(m, 2, \mathbb{R}),
 C \in M(2, m, \mathbb{R}), D \in M(2, 2, \mathbb{R})$.
The automorphism group action is given in the following explicit way. Let $Z = (z_1, \cdots, z_m) \in D^{IV}_m$ and  $T =\begin{bmatrix}
A & B\\
C & D\\
\end{bmatrix} \in {\rm Aut}(D^{IV}_{n}).$ Write $Z'=\left(\frac{1+ \frac{1}{2}Z Z^t}{\sqrt{2}}, \frac{1- \frac{1}{2}Z Z^t}{\sqrt{-2}}\right)$. Then the action of $T$ on $D^{IV}_m$ is  given by  $$T (Z) = \frac{Z A +  Z' C}{\left( Z B + Z' D \right) \left(1/\sqrt{2}, \sqrt{-1/2}\right)^t}.$$
Rephrasing in homogenous coordinates, if the holomorphic automorphism maps $Z=(z_1, \cdots, z_m) \in D^{IV}_m$ to $W=(w_1, \cdots, w_m) \in D^{IV}_m$, then there exists $T \in {\rm Aut}(D^{IV}_{m})$ such that 
$$\left[w_1, \cdots, w_m, \frac{1+\frac{1}{2}WW^t}{\sqrt{2}},  \frac{1-\frac{1}{2}WW^t}{\sqrt{-2}}\right] = \left[ z_1, \cdots, z_m,  \frac{1+\frac{1}{2}ZZ^t}{\sqrt{2}},  \frac{1-\frac{1}{2}ZZ^t}{\sqrt{-2}}\right] \cdot T.$$
 In other words, there exists nonzero $\lambda \in \mathbb{C}$, such that 
 $$\left(w_1, \cdots, w_m, \frac{1+\frac{1}{2}WW^t}{\sqrt{2}},  \frac{1-\frac{1}{2}WW^t}{\sqrt{-2}} \right) = \lambda \left(z_1, \cdots, z_m,  \frac{1+\frac{1}{2}ZZ^t}{\sqrt{2}},  \frac{1-\frac{1}{2}ZZ^t}{\sqrt{-2}}\right)\cdot T.$$ 
Note that the isotropy group $K_0$ at the
origin is $K_0 = \left\{\begin{bmatrix}
A & 0\\
0 & D\\
\end{bmatrix} \in O(m, 2, \mathbb{R}) | {\rm det}(D)=1 \right\} \cong O(m, \mathbb{R}) \times SO(2, \mathbb{R}).$

We are now at the position to prove Theorem \ref{T1}. We first establish a result on isotropy equivalence.
\begin{theorem}\label{isotr}
Let $F: \mathbb{B}^n \rightarrow D^{IV}_{n+1}$ be a holomorphic isometric embedding satisfying $F(0)=0$ and
\begin{equation}\label{type4}
F^*\omega_{D^{IV}_{n+1}} = \omega_{\mathbb{B}^n}.\end{equation} Then $F$ is isotropically equivalent to either the map $R^{IV}_n$ in (\ref{stad}) or the map $I_{n, \theta}$ in (\ref{litchi11}) for some $\theta \in [0, \pi/4)$.
\end{theorem}

\begin{proof}
{\bf First normalization:}
Write $F=(f_1, \cdots, f_{n+1})$.
By the isometry assumption, a standard reduction yields that
\begin{equation}
\sum_{i=1}^{n+1}|f_{i}|^2-\frac{1}{4}|\sum_{i=1}^{n+1} f_{i}^2|^2=\sum_{i=1}^n |z_{i}|^2.
\end{equation}
We then have \begin{equation}\label{albert11}
\left( z_1, \cdots, z_n, \frac{1}{2}\sum_{j=1}^m f^2_j(z) \right)  = (f_1(z), \cdots, f_{n+1}(z)) {\bf U}
\end{equation}
for some  $(n+1) \times (n+1)$ unitary matrix ${\bf U}$
by a lemma of D'Angelo (\cite{D2}).
Write ${\bf U} = ({\bf u_1}, \cdots, {\bf u_{n+1}})$, where each ${\bf u_i}$ for $1 \leq i \leq n+1$, is a column vector in $\mathbb{C}^{n+1}$. Write the first $n$-columns of $U$ as ${\bf U_0} = ({\bf u_1}, \cdots, {\bf u_{n}})$. By (\ref{albert11}), we have
\begin{equation}\label{albert12}
(z_1, \cdots, z_n) = (f_1(z), \cdots, f_n(z)){\bf U_0}.
\end{equation}
By the singular value decomposition of symmetric matrices, there exists an unitary $n \times n$ matrix ${\bf V}$ such that $$\bf V^t {\bf U_0}^t {\bf U_0} {\bf V} = {\rm diag}\{\lambda_1, \cdots, \lambda_n\}
$$ for real numbers $\lambda_i$ satisfying $\lambda_1 \geq \cdots \geq \lambda_n \geq 0$. Apply the unitary change of coordinate in $\mathbb{C}^n$ by letting $\widetilde z = (\widetilde z_1, \cdots, \widetilde z_n) = z \cdot {\bf V}$, where $z = (z_1, \cdots, z_n)$ or equivalently, $z = \widetilde z \cdot {\bf V}^{-1}$. By (\ref{albert12}), we have
$$(\widetilde z_1, \cdots, \widetilde z_n) = \left(f_1(\widetilde z {\bf V}^{-1}), \cdots, f_{n+1}(\widetilde z {\bf V}^{-1})\right) {\bf \hat U_0},$$ where ${\bf \hat U_0}= \bf U_0 V$ with ${\bf \hat U_0^t} {\bf \hat U_0} = {\rm diag}\{\lambda_1, \cdots, \lambda_n\}$. Moreover, by (\ref{albert11}), we have
\begin{equation}\label{albert13}
\left(\widetilde z_1, \cdots, \widetilde z_n, \frac{1}{2}\sum_{i=1}^{n+1} f^2_i(\widetilde z {\bf V}^{-1}) \right) =  \left(f_1(\widetilde z {\bf V}^{-1}), \cdots, f_{n+1}(\widetilde z {\bf V}^{-1})\right) ({\bf \hat U_0}, {\bf u_{n+1}}).
\end{equation}
Let $\widetilde F(\widetilde z) = F(\widetilde z {\bf V}^{-1})$ and write  $\widetilde F(\widetilde z) = (\widetilde f_1(\widetilde z), \cdots, \widetilde f_{n+1}(\widetilde z))$. Then equation (\ref{albert13}) can be rewritten as
\begin{equation}\label{albert14}
\left(\widetilde z_1, \cdots, \widetilde z_n, \frac{1}{2}\sum_{i=1}^{n+1} {\widetilde f}^2_i(\widetilde z) \right) =  \left(\widetilde f_1(\widetilde Z), \cdots, \widetilde f_{n+1}(\widetilde z)\right) ({\bf \hat U_0}, {\bf u_{n+1}}).
\end{equation}
Note that $({\bf \hat U_0}, {\bf u_{n+1}})$ is still an unitary $(n+1) \times (n+1)$ matrix. It follows from (\ref{albert14}) that $\widetilde F$ is also a holomorphic isometry from $\mathbb{B}^n$ to $D^{IV}_{n+1}$ and moreover, $\widetilde F$ is equivalent to $F$.
\medskip

{\bf Second normalization:}
In the following, we write ${\bf x} \cdot {\bf y}=\sum_{i=1}^k x_{i}y_{i}$ for two
$k-$dimensional vectors ${\bf x}=(x_{1},...,x_{k}), {\bf y}=(y_{1},...,y_{k}).$ We now consider the new map $\widetilde F$ in the new holomorphic coordinate $\widetilde z$. But for the simplicity of notations, we still use $F, z, {\bf U_0}$ to denote $\widetilde{F}, \widetilde{z}, {\bf\hat U_0}$ respectively. Therefore, we have
\begin{equation}\notag
\left( z_1, \cdots, z_n, \frac{1}{2}\sum_{j=1}^m f^2_j(z) \right)  = (f_1(z), \cdots, f_{n+1}(z)) ({\bf U_0}, \bf u_{n+1})
\end{equation}
with
\begin{equation}\label{albert15}
{\bf U_0^t} {\bf U_0} = {\rm diag}\{\lambda_1, \cdots, \lambda_n\}.
\end{equation}
Write $\bf U_0 = (u_1, \cdots, u_n)$. It follows from (\ref{albert15}) that
\begin{equation}\label{albert16}
\begin{split}
{\bf u_i \cdot u_i} &= \lambda_i \in \mathbb{R}~{\rm for}~ 1 \leq i \leq n, \\
{\bf u_i \cdot u_j} &=0 ~{\rm for}~ 1 \leq i \not=j \leq n.
\end{split}
\end{equation}
Note $\bf \overline U_0^t U_0 = I_n$. We have
\begin{equation}\label{albert17}
\begin{split}
{\bf u_i \cdot \overline{u_i}} &= 1 ~{\rm for}~ 1 \leq i \leq n, \\
{\bf u_i \cdot \overline{u_j}} &=0 ~{\rm for}~ 1 \leq i \not=j \leq n.
\end{split}
\end{equation}
Write ${\bf u_i = a_i} + \sqrt{-1} \bf b_i$ for $1 \leq i \leq n$. It follows from (\ref{albert16})-(\ref{albert17}) that
\begin{equation}\label{albert18}
\begin{split}
{\bf a_i \cdot {b_j}} &= 0 ~{\rm for}~ 1 \leq i, j \leq n, \\
{\bf a_i \cdot a_j} &=0 ~{\rm for}~ 1 \leq i \not=j \leq n, \\
{\bf b_i \cdot b_j} &=0 ~{\rm for}~ 1 \leq i \not=j \leq n, \\
{\bf a_i \cdot a_i + b_i \cdot b_i} &= 1~{\rm for}~ 1 \leq i \leq n.
\end{split}
\end{equation}
Therefore, these $2n$ vectors $\{\bf a_i, b_i\}_{1\leq i \leq n}$ in $\mathbb{R}^{n+1}$ are mutually orthogonal. This implies that at least $n-1$ of them are zero vectors. However, by the last equation in (\ref{albert18}), $\bf a_i$ and $\bf b_i$ cannot be both zero for each $1 \leq i \leq n$. Hence, by applying again the unitary change of coordinates in $\mathbb{C}^n$ if necessary, we assume that for each $1 \leq j \leq n-1$, either $\bf a_j$ or $\bf b_j$ is zero. Furthermore, for each fixed $1 \leq j \leq n-1$, by applying the unitary change of coordinates
$$(z_1, \cdots, z_{j-1}, z_j, z_{j+1} \cdots, z_n) \rightarrow (z_1, \cdots, z_{j-1}, \sqrt{-1}z_j, z_{j+1} \cdots, z_n)$$
 in $\mathbb{C}^n$ if necessary, we can always assume that ${\bf b_j} =0, {\bf a_j} \not=0$ for $1 \leq j \leq n-1$. Therefore, we have $\bf u_j = a_j \in \mathbb{R}^{n+1}$ for all $1 \leq j \leq n-1$, and moreover, $$(\bf u_1, \cdots, u_{n-1})^t (\bf u_1, \cdots, u_{n-1}) = I_{n-1}.$$ Extend $\{\bf u_1, \cdots, u_{n-1}\}$ to an orthonormal basis $\{\bf u_1, \cdots, u_{n-1}, c_n, c_{n+1}\}$ of $\mathbb{R}^{n+1}$ and write the $(n+1)\times(n+1)$ matrix $$\bf C = (u_1, \cdots, u_{n-1}, c_n, c_{n+1}).$$
It follows that ${\bf C},  {\bf C}^t$ are orthogonal matrices $SO(n+1)$ as ${\bf C} {\bf C}^t = \bf C^t {\bf C} = {\bf I_{n+1}}$.
Define $\widetilde F = (\widetilde f_1, \cdots, \widetilde f_{n+1}) = F \cdot (\bf C^t)^{-1}$ or equivalently $F = \widetilde F \cdot \bf C^t$. One can easily check that $\widetilde F$ is still a holomorphic isometry from $\mathbb{B}^n$ into $D^{IV}_{n+1}$ and $\widetilde F$ is equivalent to $F$. Furthermore, one has
$$\left(z_1, \cdots, z_n, \frac{1}{2} \sum_{j=1}^{n+1} {\widetilde f}^2_j\right) = (\widetilde f_1, \cdots, \widetilde f_{n+1})  \bf C^t  (\bf{U_0, u_{n+1}}).$$
Since $\bf C$ is an orthogonal matrix, then $${\bf C^t} \bf U_0 = (\bf X, \widetilde u_n)_{(n+1) \times n},$$ where $\bf X = \begin{bmatrix} \bf I_{n-1}\\ \bf 0_{2 \times (n-1)}\end{bmatrix}$ and $ \bf 0_{2 \times (n-1)}$ is the $2 \times (n-1)$ zero matrix. Note $$\overline{({\bf C^t} \bf U_0)^t} ({\bf C}^t {\bf U_0}) ={\bf \overline U_0^t} {\bf U_0} = {\bf I_n},$$ i.e. the columns of ${\bf C}^t {\bf U_0}$ are mutually orthonormal in $\mathbb{C}^{n+1}$. Then we conclude that $$\bf\widetilde u_n = \begin{bmatrix} 0, \cdots, 0, \xi_1, \xi_2 \end{bmatrix}^t$$ for $\xi_1, \xi_2 \in \mathbb{C}$ with $|\xi_1|^2 + |\xi_2|^2=1.$ 

\medskip

Note that $\bf C^t {\bf (U_0, u_{n+1})}$ is an $(n+1) \times (n+1)$ unitary matrix.
We will again use $F, {\bf U}$ to denote $\widetilde F, \bf C^t {\bf (U_0, u_{n+1})}$ respectively for the simplicity of notations. To summarize the above, we have normalize the original holomorphic isometry to the map $F$ satisfying
\begin{equation}\label{albert129}
\left(z_1, \cdots, z_n, \frac{1}{2} \sum_{j=1}^{n+1} f^2_j(z) \right) = \left(f_1(z), \cdots, f_{n+1}(z)\right) \bf U,
\end{equation} where $\bf U$ is $(n+1)\times(n+1)$ unitary matrix and $$\bf U = (X, u_n, u_{n+1})$$ for $\bf X = \begin{bmatrix} \bf I_{n-1}\\ \bf 0_{2 \times (n-1)}\end{bmatrix}$, $\bf u_n = \begin{bmatrix} 0\\ \cdots \\ 0\\ \xi_1\\ \xi_2 \end{bmatrix}$ and $\bf u_{n+1} = \begin{bmatrix} 0\\ \cdots \\ 0\\ \eta_1\\ \eta_2 \end{bmatrix}$ with $$|\xi_1|^2+|\xi_2|^2 =1, |\eta_1|^2 + |\eta_2|^2 =1 ~{\rm and}~\xi_1 \bar\eta_1 + \xi_2 \bar\eta_2 =0.$$

\bigskip

Now Replacing $F$ by $\widetilde F = e^{-\sqrt{-1}\alpha} F$ and writing $\widetilde F =(\widetilde f_1, \cdots, \widetilde f_{n+1})$, we have:
 $$
\left(z_1, \cdots, z_n, \frac{1}{2} \sum_{j=1}^{n+1} {\widetilde f}^2_j(z) \right) = \left(\widetilde f_1(z), \cdots, \widetilde f_{n+1}(z)\right) \cdot \begin{bmatrix} {e^{\sqrt{-1}\alpha} \bf I_{n-1}} & \bf 0_{(n-1) \times 1} & \bf 0_{(n-1) \times 1} \\
{\bf 0_{(n-1) \times 1}^t} &  e^{\sqrt{-1}\alpha} \xi_1 & e^{-\sqrt{-1}\alpha} \eta_1  \\
{\bf 0_{(n-1) \times 1}^t} & e^{\sqrt{-1}\alpha} \xi_2 &  e^{-\sqrt{-1}\alpha} \eta_2
\end{bmatrix}.$$
Choose a suitable $\alpha$ such that the real and imaginary parts of  $(e^{-\sqrt{-1}\alpha} \eta_1, e^{-\sqrt{-1}\alpha} \eta_2)^t$ are orthoganal. Applying the unitary change of coordinates $\hat z = (\hat z_1, \cdots, \hat z_n) = e^{-\sqrt{-1}\alpha} (z_1, \cdots, z_n)$ and defining $\hat F(\hat z) =(\hat f_1(\hat z), \cdots, \hat f_{n+1}(\hat z)) = \widetilde F(e^{\sqrt{-1}\alpha} \hat z)$, then one can easily check that $\hat{F}$ satisfies:
\begin{equation}\label{litchi12}
\left(\hat z_1, \cdots, \hat z_n, \frac{1}{2} \sum_{j=1}^{n+1} {\hat f}^2_j(\hat Z) \right) = \left(\hat f_1(\hat z), \cdots, \hat f_{n+1}(\hat z)\right) \cdot \begin{bmatrix} { \bf I_{n-1}} & \bf 0_{(n-1) \times 1} & \bf 0_{(n-1) \times 1} \\
{\bf 0_{(n-1) \times 1}^t} &   \xi_1 & e^{-\sqrt{-1}\alpha} \eta_1  \\
{\bf 0_{(n-1) \times 1}^t} &  \xi_2 &  e^{-\sqrt{-1}\alpha} \eta_2
\end{bmatrix}.\end{equation}
We will still use $F, z, \eta_i$ to denote $\hat F, \hat z, e^{-\sqrt{-1}\alpha} \eta_i$ for $i=1, 2$ respectively. Then (\ref{litchi12}) reads
$$
\left(z_1, \cdots, z_n, \frac{1}{2} \sum_{j=1}^{n+1}  f^2_j(z) \right) = \left(f_1(z), \cdots, f_{n+1}(z)\right) \begin{bmatrix} { \bf I_{n-1}} & \bf 0_{(n-1) \times 1} & \bf 0_{(n-1) \times 1} \\
{\bf 0_{(n-1) \times 1}^t} &   \xi_1 & \eta_1  \\
{\bf 0_{(n-1) \times 1}^t} &  \xi_2 &   \eta_2
\end{bmatrix}$$
with
\begin{equation}\label{ecndn}
|\xi_1|^2+|\xi_2|^2 =1, |\eta_1|^2 + |\eta_2|^2 =1 ~{\rm and}~\xi_1 \bar\eta_1 + \xi_2 \bar\eta_2 =0, {\mathrm Re}(\eta_{1},\eta_{2})^t \perp {\mathrm Im}(\eta_{1}, \eta_{2})^t.
\end{equation}
By applying an automorphism $$\begin{bmatrix} {\bf I_{n-1}} & \bf 0_{(n-1) \times 2} \\ \bf 0_{2 \times (n-1)} & \bf V_{2 \times 2} \end{bmatrix}$$ of $D^{IV}_{n+1}$ with a suitable ${\bf V} \in O(2)$, we can further make the real part of vector $(\eta_{1}, \eta_{2})$ be of form $(c, 0)$ for some $c \in \mathbb{R}.$ Consequently, by (\ref{ecndn})  we conclude that $\eta_{1} \in \mathbb{R}$ and $\eta_2 = \sqrt{-1} \eta$ for some $\eta \in \mathbb{R}$. By further applying an automorphism $$\begin{bmatrix} {\bf I_{n-1}} & \bf 0_{(n-1) \times 1} & \bf 0_{(n-1) \times 1} \\
 \bf 0_{1 \times (n-1)} &\pm 1 &0\\
 \bf 0_{1\times(n-1)} &0 &\pm 1 \end{bmatrix}$$ of $D^{IV}_{n+1}$ if necessary, we can assume that $\eta_1 \geq 0$ and $\eta \geq 0.$ Since $\eta_1^2 + \eta^2 = |\eta_1|^2 + |\eta_2|^2 =1$, write $\eta_1 = \cos\theta, \eta = \sin\theta$ for $\theta \in [0, \pi/2]$ and then $\eta_2 = \sqrt{-1} \sin\theta$. As $|\xi_1|^2+|\xi_2|^2 =1, \xi_1 \bar\eta_1 + \xi_2 \bar\eta_2 =0$, write $\xi_1 = \sqrt{-1}\sin\theta e^{\sqrt{-1}\alpha}, \xi_2 = \cos\theta e^{\sqrt{-1}\alpha}$
for $\alpha \in [0, 2\pi)$. By applying an unitary transform $(\widetilde z_1, \cdots, \widetilde z_{n-1}, \widetilde z_n) = (z_1, \cdots,  z_{n-1}, e^{\sqrt{-1}\alpha} z_n)$ in $\mathbb{C}^n$ if necessary, we may let $\xi_1 = \sqrt{-1} \sin\theta, \xi_2 = \cos\theta$. Therefore, we have normalized the map $F=(f_1, \cdots, f_{n+1})$ to make it satisfy
 $$
\left(z_1, \cdots, z_n, \frac{1}{2} \sum_{j=1}^{n+1} f^2_j(z) \right) = \left(f_1(z), \cdots, f_{n+1}(z)\right) \begin{bmatrix} {\bf I_{n-1}} & \bf 0_{(n-1) \times 1} & \bf 0_{(n-1) \times 1} \\
{\bf 0_{(n-1) \times 1}^t} &  \sqrt{-1} \sin\theta & \cos\theta  \\
{\bf 0_{(n-1) \times 1}^t} & \cos\theta &  \sqrt{-1} \sin\theta
\end{bmatrix}$$ for $\theta \in [0, \pi/2]$. Denote the matrix \begin{equation}\label{same}
\bf U = \begin{bmatrix} {\bf I_{n-1}} & \bf 0_{(n-1) \times 1} & \bf 0_{(n-1) \times 1} \\
{\bf 0_{(n-1) \times 1}^t} &  \sqrt{-1} \sin\theta & \cos\theta  \\
{\bf 0_{(n-1) \times 1}^t} & \cos\theta &  \sqrt{-1} \sin\theta
\end{bmatrix}\end{equation}
We the proceed in two different cases.

{\bf Case I:} If $\theta = \pi/4.$ In this case, we have
\begin{equation}\label{linearst}
\left(z_1, \cdots, z_n, \frac{1}{2} \sum_{j=1}^{n+1} f^2_j(z) \right) = \left(f_1(z), \cdots, f_{n+1}(z)\right) {\bf U} ,
\end{equation}
where
$${\bf U}=\begin{bmatrix} {\bf I_{n-1}} & \bf 0_{(n-1) \times 1} & \bf 0_{(n-1) \times 1} \\
{\bf 0_{(n-1) \times 1}^t} &  \frac{\sqrt{-2}}{2} & \frac{\sqrt{2}}{2}  \\
{\bf 0_{(n-1) \times 1}^t} & \frac{\sqrt{2}}{2} &  \frac{\sqrt{-2}}{2}
\end{bmatrix}.$$
By replacing $F$ by  $$F\cdot \begin{bmatrix} -\sqrt{-1}{\bf I_{n-1}} & \bf 0_{(n-1) \times 1} & \bf 0_{(n-1) \times 1} \\
 \bf 0_{1 \times (n-1)} &\sqrt{-1} &0\\
 \bf 0_{1\times(n-1)} &0 &-\sqrt{-1} \end{bmatrix},$$
and then apply the unitary transformation in $\mathbb{C}^n:(\widetilde{z_{1}},\cdots,\widetilde{z_{n}})=
(-\sqrt{-1}z_{1},\cdots, -\sqrt{-1}z_{n-1}, \sqrt{-1}z_{n}),$ we are able to make

$${\bf U}=\begin{bmatrix} {\bf I_{n-1}} & \bf 0_{(n-1) \times 1} & \bf 0_{(n-1) \times 1} \\
{\bf 0_{(n-1) \times 1}^t} &  -\frac{\sqrt{-2}}{2}& \frac{\sqrt{-2}}{2}   \\
{\bf 0_{(n-1) \times 1}^t} & \frac{\sqrt{2}}{2} & \frac{\sqrt{2}}{2}
\end{bmatrix}.$$
By solving this linear system (\ref{linearst}), one obtains the map $R^{IV}_n$ in (\ref{stad}). We note it is equivalent to the map
in (\ref{eqncl1}) by replacing $(f_{1},\cdots, f_{n+1})$ with $(f_{1},\cdots, f_{n}, -f_{n+1}).$

\bigskip

{\bf Case II:} If $\theta \in [0, \pi/2]$ with $\theta \neq \pi/4.$ The conclusion is trivial if $\theta \in [0, \pi/4).$
In the sequel, we assume   $\theta \in (\pi/4, \pi/2].$ Write $\beta = \pi/2 -\theta \in [0, \pi/4)$. Then $$\bf U = \begin{bmatrix} {\bf I_{n-1}} & \bf 0_{(n-1) \times 1} & \bf 0_{(n-1) \times 1} \\
{\bf 0_{(n-1) \times 1}^t} &  \sqrt{-1} \cos\beta & \sin\beta  \\
{\bf 0_{(n-1) \times 1}^t} & \sin\beta &  \sqrt{-1} \cos\beta
\end{bmatrix}.$$ By applying the automorphism $$(\widetilde w_1, \cdots, \widetilde w_{n+1}) = (w_1, \cdots, w_{n+1}) \begin{bmatrix} {\bf I_{n-1}} & \bf 0_{(n-1) \times 1} & \bf 0_{(n-1) \times 1} \\
{\bf 0_{(n-1) \times 1}^t} & 0 & 1 \\
{\bf 0_{(n-1) \times 1}^t} & 1 & 0
\end{bmatrix}$$ of $D^{IV}_{n+1}$, we may let $$\bf U = \begin{bmatrix} {\bf I_{n-1}} & \bf 0_{(n-1) \times 1} & \bf 0_{(n-1) \times 1} \\
{\bf 0_{(n-1) \times 1}^t} &  \sin\beta & \sqrt{-1} \cos\beta  \\
{\bf 0_{(n-1) \times 1}^t} & \sqrt{-1} \cos\beta &  \sin\beta
\end{bmatrix}.$$
Applying the automorphism $(\widetilde w_1, \cdots, \widetilde w_{n+1}) =
-\sqrt{-1} (w_1, \cdots, w_{n+1})$ of $D^{IV}_{n+1}$ and then applying the unitary transform $(\widetilde z_1, \cdots, \widetilde z_{n-1}, \widetilde z_n) = (\sqrt{-1}z_1, \cdots, \sqrt{-1}z_{n-1}, -z_n)$ of $\mathbb{C}^n$, we may let $$\bf U = \begin{bmatrix} {\bf I_{n-1}} & \bf 0_{(n-1) \times 1} & \bf 0_{(n-1) \times 1} \\
{\bf 0_{(n-1) \times 1}^t} & \sqrt{-1} \sin\beta &  \cos\beta  \\
{\bf 0_{(n-1) \times 1}^t} &-\cos\beta & -\sqrt{-1} \sin\beta
\end{bmatrix}.$$
Finally applying the automorphism $(\widetilde w_1, \cdots, \widetilde w_{n-1}, \widetilde w_{n+1}) = (w_1, \cdots, w_{n-1}, -w_{n+1})$ of $D^{IV}_{n+1}$, $$\bf U = \begin{bmatrix} {\bf I_{n-1}} & \bf 0_{(n-1) \times 1} & \bf 0_{(n-1) \times 1} \\
{\bf 0_{(n-1) \times 1}^t} & \sqrt{-1} \sin\beta &  \cos\beta  \\
{\bf 0_{(n-1) \times 1}^t} &\cos\beta & \sqrt{-1} \sin\beta
\end{bmatrix}.$$ This is the matrix in (\ref{same}). By solving the system $$
\left(z_1, \cdots, z_n, \frac{1}{2} \sum_{j=1}^{n+1} f^2_j(z) \right) = \left(f_1(z), \cdots, f_{n+1}(z)\right)  \bf U,$$ we obtain that $F$ is equivalent to $I_{n, \beta}$  in (\ref{litchi11}) for some $\beta \in [0, \pi/4)$. This establishes Theorem \ref{isotr}.
\end{proof}


\bigskip

\begin{theorem}\label{irine}
For any $\theta \in [0, \pi/4)$, $I_{n, \theta}: \mathbb{B}^n \rightarrow D^{IV}_{n+1}$ given in (\ref{litchi11})  is equivalent to $I_{n, 0}$.
\end{theorem}

\begin{proof}
We first apply the Borel embedding to embed $\mathbb{B}^n$ as an open subset of $\mathbb{P}^n$ and $D^{IV}_{n+1}$ as an open subset of $\mathbb{Q}^{n+1} \subset \mathbb{P}^{n+2}$, 
where the Borel embedding is given by $$z = (z_1, \cdots, z_{n+1}) \in D^{IV}_{n+1} \rightarrow \left[z_1, \cdots, z_{n+1}, \frac{1+ \frac{1}{2}z z^t}{\sqrt{2}}, \frac{(1- \frac{1}{2}z z^t)}{\sqrt{-2}}\right] \in \mathbb{Q}^{n+1} \subset \mathbb{P}^{n+2}.$$ We write $[z, s] = [z_1, \cdots, z_n, s]$ to denote the homogeneous coordinates in $\mathbb{P}^n$. Then under homogeneous coordinates, $I_{n, \theta}$ is identified with
\begin{equation}\label{homo}
\begin{split}
\mathcal{I}_{n, \theta}(z, s)
&= \left[z_1, \cdots, z_{n-1}, \phi_{n, \theta}(z, s), \phi_{n+1, \theta}(z, s), \phi_{n+2, \theta}(z, s), \phi_{n+3, \theta}(z, s)  \right]
\end{split}
\end{equation} from $\mathbb{P}^{n}$ to $\mathbb{P}^{n+2}$
where \begin{equation}
\begin{split}
\phi_{n, \theta}(z, s) &=\frac{(\cos\theta s +\sqrt{-1} \sin\theta z_n) - \cos\theta \sqrt{H_{\theta}(z, s)}}{\cos(2\theta)}; \\
\phi_{n+1, \theta}(z, s) &= \frac{(-\sqrt{-1}\sin\theta s + \cos\theta z_n) + \sqrt{-1}\sin\theta \sqrt{H_{\theta}(z, s)}}{\cos(2\theta)}; \\
\phi_{n+2, \theta}(z, s) &=\frac{1+\cos(2\theta)}{\sqrt{2}\cos(2\theta)}s + \frac{\sqrt{-1}\tan(2\theta) z_n}{\sqrt{2}} -\frac{1}{\sqrt{2}\cos(2\theta)} \sqrt{H_\theta(z, s)} ; \\
\phi_{n+3, \theta}(z, s) &=\frac{\cos(2\theta)-1}{\sqrt{-2}\cos(2\theta)}s - \frac{\sqrt{-1}\tan(2\theta) z_n}{\sqrt{2}} +\frac{1}{\sqrt{2}\cos(2\theta)} \sqrt{H_\theta(z, s)};\\
H_\theta(z, s) &= s^2 + 2 \sqrt{-1} \sin(2\theta) z_n s - z^2_n - \cos(2\theta) \sum_{j=1}^{n-1} z_j^2.
\end{split}\end{equation}
In particular, $$\mathcal{I}_{n, 0}(z, s) = \left[ z_1, \cdots, z_{n-1}, s-\sqrt{H_0}, z_n, \frac{2s-\sqrt{H_0}}{\sqrt{2}}, \frac{\sqrt{H_0}}{\sqrt{-2}}  \right].$$
Let $${\bf B} = \begin{bmatrix} \bf I_{n-1} & \bf 0 & \bf 0\\ \bf 0 & \frac{\cos\theta}{\sqrt{\cos(2\theta)}} & \frac{-\sin\theta}{\sqrt{\cos(2\theta)}}\sqrt{-1} \\ \bf 0 &  \frac{\sin\theta}{\sqrt{\cos(2\theta)}}\sqrt{-1} & \frac{\cos\theta}{\sqrt{\cos(2\theta)}}
\end{bmatrix} \in U(n, 1) = \rm{Aut}(\mathbb{B}^n),$$
and define $\mathcal{\hat I}_{n, \theta}(z, s) = \mathcal{I}_{n, \theta}\left( (z, s) \cdot {\bf B} \right).$ 
Then it follows from the straightforward calculation that $$\mathcal{\hat I}_{n, \theta} =\left[z_1, \cdots, z_{n-1}, \frac{s-\cos\theta \sqrt{H_0}}{\sqrt{\cos(2\theta)}}, \frac{z_n+\sin\theta\sqrt{-H_0}}{\sqrt{\cos(2\theta)}}, \frac{2\cos\theta s - \sqrt{H_0}}{\sqrt{2\cos(2\theta)}}, \frac{-2\sin\theta z_n-\sqrt{-H_0}}{\sqrt{2\cos(2\theta)}} \right].$$
Let $T = \begin{bmatrix} \bf I_{n-1} & \bf 0_{(n-1)\times 4} \\
\bf 0_{4 \times (n-1)} & \bf V
 \end{bmatrix}$ with $${\bf V}=\frac{1}{\sqrt{\cos(2\theta)}}
\begin{bmatrix} 1-4\sin^2(\theta/2) & 0 & 2\sqrt{2}\sin^2(\theta/2) & 0 \\
0 & 1 & 0 & -\sqrt{2}\sin\theta \\
2\sqrt{2}\sin^2(\theta/2) & 0 & 1-4\sin^2(\theta/2) & 0\\
0 & -\sqrt{2}\sin\theta & 0 & 1
 \end{bmatrix}. $$
Then one can verify that $T \in {\rm Aut}(D^{IV}_{n+1})$ and
\begin{equation}\notag
\begin{split}
&~~ \mathcal{I}_{n, 0} \cdot T \\
&= \left( z_1, \cdots, z_{n-1}, s-\sqrt{H_0}, z_n, \frac{2s-\sqrt{H_0}}{\sqrt{2}}, \frac{\sqrt{H_0}}{\sqrt{-2}}  \right) \cdot T \\
&= \left(z_1, \cdots, z_{n-1}, \frac{s-\cos\theta \sqrt{H_0}}{\sqrt{\cos(2\theta)}}, \frac{z_n+\sin\theta\sqrt{-H_0}}{\sqrt{\cos(2\theta)}}, \frac{2\cos\theta s - \sqrt{H_0}}{\sqrt{2\cos(2\theta)}}, \frac{-2\sin\theta z_n - \sqrt{-H_0}}{\sqrt{2\cos(2\theta)}} \right)\\
&= \mathcal{\hat I}_{n, \theta}=\mathcal{I}_{n, \theta}\left( (z, s) \cdot {\bf B} \right). 
\end{split}\end{equation}
This implies that $I_{n, \theta}$ is equivalent to $I_{n, 0}.$
\end{proof}

Note that a rational map from $\mathbb{B}^n \rightarrow D^{IV}_{n+1}$ cannot be equivalent to an irrational map. Thus combining Theorem \ref{isotr},  \ref{irine},  we obtain the classification result Theorem \ref{T1} for holomorphic isometries from $\mathbb{B}^n$ into $D^{IV}_{n+1}.$

\begin{remark}
$n=1$ is a special case ($n+1 = 2n$).
Let $F$ be a holomorphic isometry from $\Delta \subset \mathbb{C}$ to $D^{IV}_2$ satisfying $F^*\omega_{D^{IV}_{2}} = \omega_{\Delta}$. It follows from the same argument of Theorem \ref{T1} that $F$ is either equivalent to the totally geodesic embedding $(\sqrt{2}z / 2, \sqrt{-2}z /2)$ or equivalent to the non-totally geodesic embedding $I_{1, 0}$. Note that $D^{IV}_2$ is biholomorphic to bidisc $\Delta^2 \subset \mathbb{C}^2$. In the Euclidean coordinate of $\Delta^2$, the first map is $z \rightarrow (z, 0)$ and $I_{1, 0}$ is the square root embedding constructed by Mok \cite{M4}. This classification result from $\Delta$ to $\Delta^2$ was obtained earlier by Ng (cf. \cite{Ng1} Theorem 7.1).
\end{remark}


\newpage

\section*{\center Appendix}

\bigskip

We sketch a proof for Theorem \ref{propergene1} in this Appendix. Theorem \ref{propergene1} can be proved by using the idea from \cite{BEH2} in a very similar manner. We first introduce the following notations(See more details in \cite{BH} and \cite{BEH2}). Assume $N \geq 3$ and let $\mathbb{H}^{N}_1 \subset \mathbb{C}^N$ be the hyperquadric of signature 1 defined by

$$\mathbb{H}^{N}_1 = \{(z, w=u+\sqrt{-1}v )\in \mathbb{C}^N | v= \sum_{j=1}^{N-2} |z_j|^2 - |z_{N-1}|^2\},$$ where $z=(z_1, \cdots, z_{N-1})$. Let $\mathbb{H}^n$ be the Heisenberg hypersurface

$$\mathbb{H}^{n}=\{(z, w=u+\sqrt{-1}v ) \in \mathbb{C}^n  | v= \sum_{j=1}^{n-1} |z_j|^2 \},$$ where $z=(z_1, \cdots, z_{n-1})$. To prove Theorem \ref{propergene1}, it suffices to show the following theorem.

\begin{theorem}\label{hyper}
Assume $n \geq 4, n+1 \leq N \leq 2n-2$ and let $F$ be a holomorphic map from a connected open set $U \subset \mathbb{C}^n$ containing 0 to $\mathbb{C}^N$. If $F(U \cap \mathbb{H}^n) \subset \mathbb{H}^N_1$ and $F(U) \not\subset \mathbb{H}^N_1$, then
\begin{itemize}
\item $F$ is equivalent to $(z_1, \cdots, z_{n-1}, 0, w)$ if $N=n+1$;
\item $F$ is equivalent to $( z_1, \cdots, z_n, {\bf 0}, \psi, \psi, w)$ if $N > n+1$, where $\psi$ is some holomorphic function on $U$.
\end{itemize}
\end{theorem}

To prove Theorem \ref{hyper}, we will first establish the following proposition:

\begin{proposition}\label{hyper2}
Let $F=(\tilde f, g)$ be a holomorphic map from an open connected set $U \subset \mathbb{C}^n$ containing 0 to $\mathbb{C}^N$ with $n \geq 4, n+1 \leq N \leq 2n-2$. Assume that $F(U\cap \mathbb{H}^n) \subset \mathbb{H}^N_1$, $F(0)=0$ and $F$ is CR transversal at 0. Then $\frac{\partial g}{\partial w}(0)>0$ and $F$ is
\begin{itemize}
\item equivalent to $(z_1, \cdots, z_{n-1}, 0, w)$ if $N=n+1$;
\item equivalent to $( z_1, \cdots, z_n, {\bf 0}, \psi, \psi, w)$ if $N > n+1$, where $\psi$ is some holomorphic function on $U$.
\end{itemize}
\end{proposition}

We merely prove the case $N > n+1$ in Proposition \ref{hyper2} and the other case follows similarly. \\

{\bf Proof of Proposition \ref{hyper2}.} We first show $\frac{\partial g}{\partial w}(0)>0$. Write $$F=(\tilde f, g)=(f, \varphi, g)=(f_1, \cdots, f_{n-1}, \varphi_1, \cdots, \varphi_{N-n}, g)$$ and write the tangent vector fields of $\mathbb{H}^n$ at 0
\begin{equation}\notag
\begin{split}
L_i &= \frac{\sqrt{-1}}{2} \frac{\partial}{\partial z_i} - \bar z_i \frac{\partial}{\partial w}, 1 \leq i \leq n-1, \\
T&= \frac{\partial}{\partial w} + \frac{\partial}{\partial \bar w}.
\end{split}
\end{equation}
$F(U \cap \mathbb{H}^n) \subset \mathbb{H}^N_1$ implies that \begin{equation}\label{cr1}
\rho(F, \bar F) = -\frac{g-\bar g}{2\sqrt{-1}} + |f|^2 + \sum_{i=1}^{N-n-1} |\varphi_i|^2 - |\varphi_{N-n}|^2 =0
\end{equation}
holds on $\mathbb{H}^n$. Here $|f|^2= \sum_{i=1}^{n-1} |f_i|^2.$
Applying $L^k_i$ to (\ref{cr1}) and evaluating at 0, we get
$$\frac{\partial^k g}{\partial z^k_i}(0)=0.$$
Applying $T$ to (\ref{cr1}) and evaluating at 0, we get $$\frac{\partial g}{\partial w}(0) = \overline{\frac{\partial g}{\partial w}(0)}.$$ This implies $\frac{\partial g}{\partial w}(0) = \lambda \in \mathbb{R}$. Since $F$ is transversal at 0, we have $\lambda \not= 0.$
We now write \begin{equation}\notag
\begin{split}
g &= \lambda w + O(|(z, w)|^2), \\
\tilde f_j &= b_j w + \sum_{i=1}^n a_{ij} z_i + O(|(z, w)|^2),
\end{split}
\end{equation}
where $b_j \in \mathbb{C}, a_{ij} \in \mathbb{C}$. The second equation can be rewritten as
$$\tilde f = w (b_1, \cdots, b_{N-1}) + (z_1, \cdots, z_{n-1}) A + \hat{\tilde f},$$ where $A = (a_{ij})_{(n-1) \times (N-1)}$ is an $(n-1) \times (N-1)$ matrix and $\hat{\tilde f} = O(|(z, w)|^2)$. We call a function $h$ on $U \cap \mathbb{H}^n$ has weighted degree $s$, denoted by $h \in o_{wt}(s)$, if $$\lim_{t \rightarrow 0+} \frac{h(tz,  t \bar z, t^2 u)}{|t|^s} =0.$$
For a smooth function $h(z, \bar z, u)$ defined in $U \cap \mathbb{H}^n$, we denote by $h^{(k)}(z, \bar z, u)$ the sum of terms of weighted degree $k$ in the Taylor expansion of $h$ at 0. We also denote by $h^{(k)}(z, \bar z, u)$ a weighted homogeneous polynomial of  degree $k$. When $h^{(k)}(z, \bar z, u)$ extends to a weighted holomorphic polynomial  of degree $k$, we write it as $h^{(k)}(z, w)$ or $h^{(k)}(z)$ if it depends only on $z$.

Note that (\ref{cr1}) implies that
\begin{equation}\label{cr2}
\frac{g-\bar g}{2\sqrt{-1}} = |f|^2 + \sum_{i=1}^{N-n-1} |\varphi_i|^2 - |\varphi_{N-n}|^2
\end{equation}
holds on $w=u+\sqrt{-1} |z|^2$ near 0. We collect terms of weighted degree two on both sides of (\ref{cr2}) to get
\begin{equation}
\lambda |z|^2 = z A E(N-2, N-1) \bar A^t \bar z^t,
\end{equation}
which further implies $$\lambda I_n = A E(N-2, N-1) \bar A^t.$$ Here $E(N-2, N-1)$ denotes the $(N-1) \times (N-1)$ diagonal
matrix with its first $(N-2)$ diagonal element $+1$ and the rest $-1.$ It follows from linear algebra that $\lambda >0.$

\medskip

We now fix some notations. For two $m$-tuples $x=(x_1, \cdots, x_m), y=(y_1, \cdots, y_m)$ of complex numbers, we write $<x, y>=\sum_{j=1}^m x_j y_j$ and $|x|^2 = <x, \overline{x}>$ and write $<x, y>_1=\sum_{j=1}^{m-1} x_j y_j - x_m y_m$ and $|x|_1^2 = <x, \overline{x}>_1$.

In this following context, to make notations easier, we will assume $N=2n-2.$ The proof of other cases is of no significant difference. By the same argument as in \cite{BH} (cf. Lemma 2.2), we can assume that $F$ has the following normalization
\begin{equation}\label{normal1}
\begin{split}
f(z, w) &= z+ \frac{\sqrt{-1}}{2} a^{(1)}(z) w + o_{wt}(3),\\
\varphi(z, w) &= \varphi^{(2)}(z) + o_{wt}(2),\\
g(z, w) &= w+ o_{wt}(4),
\end{split}
\end{equation}
with $$<a^{(1)}(z), \bar z> |z|^2 = <\varphi^{(2)}(z), \overline{\varphi^{(2)}(z)}>_1.$$

By Huang's lemma (cf. \cite{Hu1} Lemma 3.2 or \cite{BEH2} Lemma 2.1), $$a^{(1)}(z) \equiv 0 ~\text{and}~ <\varphi^{(2)}(z), \overline{\varphi^{(2)}(z)}>_1 \equiv 0.$$ Therefore, $$f^{(3)}\equiv 0.$$
Now suppose that we have obtained
\begin{equation}\label{cr4}
g^{(t)}\equiv 0 ~\text{and}~ f^{(t-1)}\equiv 0
\end{equation}
for $3 \leq t < s$. (Note that (\ref{cr4}) is already proved for $s=5$.) Now by collecting terms of weighted degree $s$, we obtain
\begin{equation}\label{cr5}
{\rm Im} \{g^{(s)}(z, w) - 2 \sqrt{-1} <\bar z, f^{(s-1)}(z, w)> \} = \sum_{s_1+s_2=s} <\varphi^{(s_1)}(z, w), \overline{\varphi^{(s_1)}(z, w)}>_1
\end{equation}
when $w=u +\sqrt{-1} <z, \bar z>$. We shall use the notation
$$\mathcal{L}(p, q)(z, \bar z, u) = {\rm Im} \{q(z, w) - 2\sqrt{-1}<\bar z, p(z, w)> \}|_{w = u +\sqrt{-1} <z, \bar z>},$$
where $p(z, w) = (p_1(z, w), \cdots, p_{n-1}(z, w))$ and $q(z, w)$ are holomorphic polynomials.  Equation (\ref{cr5}) can be written as
\begin{equation}\label{star}
\mathcal{L}(f^{(s-1)}, g^{(s)})(z, \bar z, u) = \sum_{s_1+s_2=s} <\varphi^{(s_1)}(z, w), \overline{\varphi^{(s_1)}(z, w)}>_1|_{w = u +\sqrt{-1} <z, \bar z>}.
\end{equation}

The following result is crucial in the proof of Proposition \ref{proper2}.

\begin{lemma}\label{huang}
Let $F=(f, \varphi, g)$ be any normalized map as above sending an open piece $M$ of $\mathbb{H}^n$ near 0 into $\mathbb{H}^N_1$ with $N \leq 2n-2$. Assume that for all $3 \leq t \leq 2(s^*-1)$,
\begin{equation}\label{cr6}
f^{(t-1)} \equiv 0, g^{(t)}\equiv 0, <\varphi^{(s_1)}, \overline{\varphi^{(s_2)}}>_1 \equiv 0
\end{equation}
for any $(s_1, s_2)$ with $s_1+s_2=t $. Then (\ref{cr6}) holds also for $t=2(s^*-1)+1$ and $t=2s^*$ for any such a map $F$.
\end{lemma}

Once Lemma \ref{huang} has been proved, we conclude by induction that (\ref{cr6}) holds for all $t \geq 3$. We shall need some notation and results from \cite{BEH2} and \cite{EHZ2}. Given a real-valued power series $A(z, \bar z, w, \bar w)$, we use the expansion
$$A(z, \bar z, w, \bar w) = \sum_{\mu \nu \gamma \delta} A_{\mu \nu \gamma \delta} (z, \bar z) w^{\gamma} \bar w^{\delta}$$ for $(z, w) \in \mathbb{C}^{n-1} \times \mathbb{C}$ where $A_{\mu \nu \gamma \delta} (z, \bar z)$ is a bihomogeneous polynomial in $(z, \bar z)$ of bidegree $(\mu, \nu)$ for every  $(\mu, \nu, \gamma, \delta) \in \mathbb{Z}^4_+$. We recall from \cite{EHZ2} that $A(z, \bar z, w, \bar w)$ belongs to the class ${\tilde S}_k$ for a positive integer $k$ if $A$ vanishes at least up to order 2 at 0 and for every $(\mu, \nu, \gamma, \delta) \in Z^4_+$, we have $$ A_{\mu \nu \gamma \delta} (z, \bar z) w^{\gamma} \bar w^{\delta} = \sum_{j=1}^k p_j(z, w) \overline{q_j(z, w)},$$ where $p_j, q_j$ are homogeneous holomorphic polynomials of the appropriate degree. We recall the following result from \cite{EHZ2}.

\begin{theorem}\label{4} \cite{EHZ2}
Let $A(z, \bar z, w, \bar w)$ be a real-valued weighted homogeneous polynomial of degree $s \geq 5$ and assume that $A \in \tilde S_{n-2}$. If $p(z, w) = (p_1(z, w), \cdots, p_{n-1}(z, w))$ and $q(z, w)$ are weighted homogeneous holomorphic polynomials of degree $s-1$ and $s$ respectively, such that $$\mathcal{L}(p, q)(z, \bar z, u) = A(z, \bar z, w, \bar w)|_{w=u+\sqrt{-1}|z|^2}.$$ Then $$p \equiv 0, q\equiv 0 ~\text{and}~ A \equiv 0.$$
\end{theorem}

\medskip

{\bf Proof of Lemma \ref{huang}.} As in \cite{BEH2}, we shall first prove that (\ref{cr6}) holds for $t=2(s^*-1)+1=2s^*-1$ with $s^* \geq 3$. Recall the hypotheses in Lemma \ref{huang} implies that (\ref{star}) holds with $s=t$. Note that the hypotheses also imply that $<\varphi^{(s_1)}, \overline{\varphi^{(s_1)}}>_1 \equiv 0$ for $2 \leq s_1 \leq s^*-1$. By a lemma of D'Angelo \cite{D2}, we conclude that there are constants $a^{s_1}_j$ such that
\begin{equation}\label{cr7}
\varphi^{(s_1)}_j = a^{s_1}_j \varphi^{(s_1)}_{n-2}, ~\text{for}~1 \leq j \leq n-3, 2 \leq s_1 \leq s^*-1.
\end{equation}

Now if $s_1+s_2 = 2 s^*-1$, then ${\rm min}(s_1, s_2) \leq s^*-1$. Without of loss of generality, assume $s_1 = {\rm min}(s_1, s_2)$. Then it follows from (\ref{cr7}) that
$$<\varphi^{(s_1)}, \overline{\varphi^{(s_2)}}>_1 = \varphi^{(s_1)}_{n-2} q^{(s_2)},$$ where $q^{(s_2)}(z, w) = \sum_{j=1}^{n-3} a^{s_1}_j \overline{\varphi^{(s_2)}_j} - \overline{\varphi^{(s_2)}_{n-2}}$.
Since $1 \leq n-2$, it follows that $A(z, \bar z, w, \bar w) : = \sum_{s_1+s_2 = 2s^*-1} <\varphi^{(s_1)}(z, w), \overline{\varphi^{(s_2)}(z, w)}>_1$ belongs to $\tilde S_{n-2}$. It follows from (\ref{star}) and Theorem \ref{4} that $f^{(s-1)} \equiv 0, g^{(s)} \equiv 0$ where $s=2s^*-1$ and $A \equiv 0$. By the definition of $A$, we have $<\varphi^{(s_1)}, \overline{\varphi^{(s_2)}}>_1 \equiv 0$ for $s_1+ s_2 = 2 s^*-1$.

It remains to show ({\ref{cr6}}) for $t=2s^* \geq 6$. We first complexity (\ref{star}) with $s=2s^*$:
\begin{equation}\label{2star}
\begin{split}
& g^{(2s^*)}(z, w) - \overline{g^{(2s^*)}(\xi, \eta)} - 2 \sqrt{-1} <\bar \xi, f^{(2s^*-1)}(z, w)> - 2 \sqrt{-1} <z, \overline{f^{(2s^*-1)}(\xi, \eta)}> \\
&= 2 \sqrt{-1} \sum_{k} <\varphi^{(k)}(z, w), \overline{\varphi^{(2s^*-k)}(\xi, \eta)}>_1.
\end{split}
\end{equation}

Write $$L_j = \frac{\partial}{\partial z_j} + 2 \sqrt{-1} \bar\xi_j \frac{\partial}{\partial w}, ~ 1 \leq j \leq n-1$$ be the tangent vector field of the complex hypersurface defined by $w=\bar \eta + 2\sqrt{-1} <z, \bar \xi>$.
We will first prove
\begin{equation}\label{cr7'}
f^{(2s^*-1)}(z, w) \equiv 0, g^{(2s^*)} \equiv 0.
\end{equation}
We begin by establishing the following proposition.

\begin{proposition}
For $1 \leq j \leq n-1$, we have
$$f_j^{2s^*-1}(z, w)=a^{(1)}_j(z) w^{s^*-1}, g^{(2s^*)}(z, w) = d^{(0)}w^{s^*} ~\text{and}~ a^{(1)}(z)= d^{(0)} z_j.$$
\end{proposition}
%
%
{\bf Proof of Proposition.} Applying $L_j$ to (\ref{2star}), we have on $w=\bar \eta + 2\sqrt{-1} <z, \bar \xi>$,
\begin{equation}\notag
\begin{split}
& L_j \left(g^{(2s^*)}(z, w)\right) - 2 \sqrt{-1} <\bar \xi, L_j f^{(2s^*-1)}(z, w)> - 2 \sqrt{-1}\overline{f^{(2s^*-1)}(\xi, \eta)} \\
&= 2 \sqrt{-1} \sum_{k} <L_j \varphi^{(k)}(z, w), \overline{\varphi^{(2s^*-k)}(\xi, \eta)}>_1.
\end{split}\end{equation}
Write \begin{equation}\notag
\begin{split}
f^{(s-1)}_j(z, w) &= \sum a_j^{(\tau_j)}(z) w^{\tau^j_s};\\
\varphi^{(k)} &= \sum b^{(\mu_k)}_k(z) w^{(\mu^*_k)};\\
g^{(s)}(z, w) &= \sum d^{(j)}(z) w^{n^j_s},
\end{split}
\end{equation}
where sums run over all indices such that $\tau_j + 2 \tau^j_s = s-1, \mu_k + 2 \mu^*_k =k, j +2n^j_s=s$ with $s=2 s^*$. Letting $w=0, \eta = 2 \sqrt{-1} <\bar z, \xi>$ and collecting terms of degree $k >2$ in $\xi$ and degree $P$ in $z$. We obtain
\begin{equation}\label{cr8}
\begin{split}
& 2\sqrt{-1} \overline{a^{(K-P)}_j} \overline{\eta^P} \\
&= 2\sqrt{-1} \sum_{k=2}^{s-2} <\varphi^{(k)}_{z_j}(z, 0), \sum \overline{b^{(\mu_{s-k})}_{s-k}(\xi)} \overline{\eta^{\mu^*_{s-k}}}>_1  - 4 \sum_{k'=3}^{s-2} \bar \xi_j <\varphi^{(k')}_w(z, 0), \sum \overline{b^{(\mu_{s-k'})}_{s-k'}(\xi)} \overline{\eta^{\mu^*_{s-k'}}}>_1,
\end{split}
\end{equation}
where the sum inside $<\cdot, \cdot>_1$ run over the indices $\mu_{s-k}+\mu_{s-k}^*=K, k-1+\mu^*_{s-k}=P$ and $\mu_{s-k'}+\mu_{s-k'}^*+1=K, k'-2+\mu^*_{s-k^*}=P$. Note that the solution only exist when $K+P+1=s=2s^*$. Letting $K=s^*+p, P=s^*-p-1$, we then get $\mu_{s-k}=k+2p, \mu^*_{s-k}=s^*-p-k$ and $\mu_{s-k'}=k'+2p-2, \mu^*_{s-k'}=s^*-p-k'+1$. Since both $\mu^*_{s-k}$ and $\mu^*_{s-k'}$ must be nonnegative, one has $2 \leq k \leq s^*-p, 3 \leq k' \leq s^*-p+1$. We now rewrite (\ref{cr8}) as
\begin{equation}\notag
\begin{split}
&-2\sqrt{-1} \overline{a^{2p+1}_j(\xi)} \overline{\eta^{s^*-p-1}} \\
&= 2\sqrt{-1} \sum_{k=2}^{s^*-p}<\varphi^{(k)}_{z_j}(z, 0), \overline{b_{s-k}^{(k+2p)}(\xi)}>_1 \overline{\eta^{s^*-p-k}} - 4 \sum_{k'=3}^{s^{*}-p+1} \bar \xi_j <\varphi^{(k')}_w(z, 0), \overline{b^{(k'+2p-2)}_{s-k'}(\xi)}>_1 \overline{\eta^{s^*+1-p-k'}}
\end{split}
\end{equation}
which can further be rewritten as
\begin{equation}\label{cr9}
\begin{split}
& -2\sqrt{-1}\overline{a_j^{(2p+1)}(\xi)} \overline{\eta^{s^*-p-1}} \\
&= \sum_{k=2}^{s^*-p} \left(2\sqrt{-1}<\varphi^{(k)}_{z_j}(z, 0), \overline{b^{(k+2p)}_{s-k}(\xi)}>_1 - 4 \bar\xi_j <\varphi^{(k+1)}_w(z, 0), \overline{b^{(k+2p-1)}_{s-k-1}(\xi)}>_1 \right) \overline{\eta^{s^*-p-k}}
\end{split}.
\end{equation}
This equation is valid for all $p=0, \cdots, s^*-1$. If $p=s^*-1$, the sum on the right hand side is void. Note that for any $q \leq s^*-1$, we can use (\ref{cr7}) to rewrite
\begin{equation}\notag
\begin{split}
<\varphi^{(q)}_{z_j}(z, 0), \overline{b^{(q+2p)}_{s-q}(\xi)}>_1 &= \sum_{i=1}^{n-3} \varphi^{(q)}_{i, z_j}(z, 0) \overline{b^{(q+2p)}_{i, s-q}(\xi)} - \varphi^{(q)}_{n-2, z_j}(z, 0) \overline{b^{(q+2p)}_{n-2, s_q}(\xi)} \\
&=\varphi^{(q)}_{n-2, z_j}(z, 0) \overline{C(\xi)}
\end{split}
\end{equation}
with $$C(\xi) = \sum_{i=1}^{n-3} a^s_j\overline{b^{(q+2p)}_{i, s-q}(\xi)} - \overline{b^{(q+2p)}_{n-2, s-q}(\xi)}.$$ We can make a similar substitution in $<\varphi^{(q)}_w(z, 0), \overline{b^{(q+2p)}_{s-k-1}(\xi)}>_1$ for any $q \leq s^*-1$. Therefore, if $2 \leq q \leq s^*-1$, we then conclude by Corollary 2.2 in \cite{BEH2}, as $2 \leq n-2$, that
\begin{equation}\label{cr10}
a^{(2p+1)}_j(z) \equiv 0, ~\text{for}~ j=1, \cdots, n-1 ~\text{and}~  p=2, \cdots, s^*-1.
\end{equation}
When $p=1$, (\ref{cr9}) can be written as
\begin{equation}\label{cr11}
\begin{split}
-2\sqrt{-1} \overline{a^{(3)}_j(\xi)} \overline{\eta^{s^*}} &= \sum_{k=1}^{s^*-2} \left( 2\sqrt{-1} <\varphi^{(k)}_{z_j}(z, 0), \overline{b^{(k+2)}_{s-k}(\xi)}>_1 - 4 \bar\xi_j <\varphi^{(k+1)}_w(z, 0), \overline{b^{(k+1)}_{s-k-1}(\xi)}>_1  \right) \overline{\eta^{s^*-1-k}} \\
&+ 2\sqrt{-1} <\varphi^{(s^*-1)}_{z_j}(z, 0), \overline{b^{(s^*+1)}_{s^*+1}(\xi)}>_1 -4 \bar\xi_j <\varphi^{(s^*)}_w(z, 0), \overline{b^{(s^*)}_{s^*}(\xi)}>_1
\end{split}.
\end{equation}
We now turn to the equation (\ref{2star}) in which we set $w=0$ and $\eta = 2 \sqrt{-1} <\bar z, \xi>$. Collecting terms of degree $s^*$ in $z$ and $s^*$ in $\xi$, we obtain
$$-\overline{d^{(0)}\eta^{s^*}} - 2 \sqrt{-1}<z, \overline{a^{(1)}(\xi)}> \overline{\eta^{s^*-1}} = 2 \sqrt{-1} \sum_{k=2}^{s^*} <\varphi^{(k)}(z, 0), \overline{b^{(k)}_{s-k}(\xi)}>_1 \overline{\eta^{s^*-k}}.$$
This can be rewritten as $$-\overline{d^{(0)}\eta^{s^*}} - 2 \sqrt{-1}<z, \overline{a^{(1)}(\xi)}> \overline{\eta^{s^*-1}} -2 \sqrt{-1} \sum_{k=2}^{s^*-1} <b^{(k)}_k(z), \overline{b^{(k)}_{s-k}(\xi)}> \overline{\eta^{s^*-k}} =2 \sqrt{-1} <b^{(s^*)}_{s^*}(z), \overline{b^{(s^*)}_{s^*}(\xi)}>_1 .$$
Note the left hand side is divisible by $\eta$ that is a summation of $n-1$ terms, while $<b^{(s^*)}_{s^*}(z), \overline{b^{(s^*)}_{s^*}(\xi)}>_1$ is a summation of $n-2$ terms, by Huang's lemma (cf. Lemma 3.2 in \cite{Hu1}  or Lemma 2.1 in \cite{BEH2} ), we have $$<b^{(s^*)}_{s^*}(z), \overline{b^{(s^*)}_{s^*}(\xi)}>_1 \equiv 0.$$
Therefore $$b^{s^*}_{j s^*}(z) = A^{s^*}_j b^{(s^*)}_{(n-2) s^*}(z), ~ 1 \leq j \leq n-3 $$ for some constant $A^{s^{*}}_j$. Here $b^{(s^*)}_{j s^*}$ is the $j$-th component of the vector-valued function $b^{(s^*)}_{s^*}(z)$. Using Corollary 2.2 in \cite{BEH2}, we have
\begin{equation}\label{cr12}
<b^{(k)}_k(z), \overline{b^{(k)}_{s-k}(\xi)}>_1 \equiv 0 ~\text{for}~ k=2, \cdots, s^*-1
\end{equation}
and \begin{equation}\label{cr13}
\overline{d^{(0)}\eta} = 2 \sqrt{-1} <z, \overline{a^{(1)}(\xi)}>.
\end{equation}
Equation (\ref{cr13}) implies $$a^{(1)}_j(z) = d^{(0)} z_j.$$
If we use (\ref{cr12}) to substitute for $b^{(s^*)}_{j s^*}$ in (\ref{cr11}) as we did before, then it follows from Corollary 2.2 in \cite{BEH2} again that
\begin{equation}\label{cr14}
a^{(3)}_j \equiv 0 ~\text{for}~ j =1, \cdots, n-1.
\end{equation}
(\ref{cr10}) and (\ref{cr14}) imply that $$f^{2s^*-1}(z, w)  = a^{(1)}(z) w^{s^*-1}.$$
To show $g^{(2s^*)}(z, w) = d^{(0)}w^{s^*}$, we go back to (\ref{2star}) in which we again set $w=0$ and $\eta = -2\sqrt{-1}<\bar z, \xi>$. Note that we have proved that the degree of $\overline{f^{2s^*-1}(\xi, \eta)}$ in $\xi$ is $s^*$. Thus, if we collect terms of degree $K=s^*+p$ in $\xi$ with $p \geq 1$, then we obtain $$-\overline{d^{(2p)}(\xi)\eta^{s^*-p}} = 2 \sqrt{-1} \sum_{k=2}^{s^*-p} <\varphi^{(k)}(z, 0), \sum \overline{b^{(\mu_{s-k})}_{s-k}(\xi)} \overline{\eta^{\mu^*_{s-k}}}>_1 ,$$
where the sum inside $<\cdot, \cdot>_1$ rums over $\mu_{s-k} + \mu^*_{s-k} = s^*+p, \mu_{s-k} + 2 \mu^*_{s-k} = s-k$. As before, the equation above can be rewritten as
$$-\overline{d^{(2p)}(\xi) \eta^{s^*-p}} = 2 \sqrt{-1} \sum_{k=2}^{s^*-p} <\varphi^{(k)}(z, 0), \overline{b^{(k+2p)}_{s-k}(\xi)}> \overline{\eta^{s^*-p-k}} ~\text{for}~p=1, \cdots, s^*.$$
We then substitute for $\varphi^{(k)}(z, 0)$ using (\ref{cr7}) and apply Corollary 2.2 in \cite{BEH2} to conclude
$$d^{(2p)}(\xi) \equiv 0 ~\text{for}~ p=1, \cdots, s^*.$$ This establishes the proposition.

We will need the following lemma to complete the proof of $(\ref{cr7'})$.

\begin{lemma}
Suppose that the hypotheses in Lemma \ref{huang} hold. Assume further that for any $p \in M$, we have
\begin{equation}\notag
\begin{split}
f^{**}_p(z, w) &= z+a^{(1)}_p(z) w^{s^*-1} + O_{wt}(s) \\
g^{**}_p(z, w) &= w+ d^{(0)}_p w^{s^*} + O_{wt}(s+1)
\end{split}
\end{equation}
where $s=2s^*$. Then $a^{(1)}_p(z) \equiv 0, d^{(0)}_p=0$ and hence $$f(z, w)=z+O_{wt}(s), g(z, w)=w+ O_{wt}(s+1).$$
\end{lemma}

 \begin{proof}
 One can apply the identical proof of Lemma 3.5 in \cite{BEH2}. We need to employ the moving points trick that was first introduced in \cite{Hu1} and has been extensively used in many literatures (cf. \cite{Hu1}, \cite{BH}, \cite{BEH2}).
 \end{proof}

To complete the induction step in the proof of Lemma \ref{huang}, we must show that $$<\varphi^{(k)}(z, w), \overline{\varphi^{(s-k)}(\xi, \eta)}>_1 \equiv 0 ~\text{for}~ k=2, \cdots, s-2$$ with $s=2s^*$. This can be established by a very similar proof as \cite{BEH2} (cf. pp 1649-1655) using the same idea as above with the following difference.
In our setting, one should apply Huang's lemma (cf. Lemma 3.2 \cite{Hu1}  or Lemma 2.1 \cite{BEH2} ) while \cite{BEH2} uses Lemma 2.3.
Here note Huang's lemma can be applied as by our assumption $<\varphi, \bar\varphi>_1$ is a summation with less terms than $<z, \bar \xi>$, i.e. $N-n < n-1$.
\qed

\medskip

Now we are able to finish the proof of Proposition \ref{hyper2}. First we may assume that $F$ satisfies the normalization (\ref{normal1}). By Lemma \ref{huang}, we conclude that
$$f(z, w)\equiv z, g(z, w)\equiv w, <\varphi(z, w), \overline{\varphi(\xi, \eta)}>_1 \equiv 0.$$
By the result of D'Angelo \cite{D2}, there is a constant $(n-3)\times(n-3)$ unitary matrix $U$ such that
$$(\varphi_1, \cdots, \varphi_{n-3}) \cdot U = (\varphi_{n-2}, 0, \cdots, 0)$$ with $(n-4)$ zero components on the right hand side of the equation. Letting $\gamma$ be the automorphism of $\mathbb{H}^N_1$ given by $$\gamma(z, w) = (z_1, \cdots, z_{n-1}, (z_n, \cdots, z_{N-2}) \cdot U, z_{N-1}, w),$$ then $\gamma \circ F$ satisfies the conclusion of Proposition \ref{hyper2}. This finishes the proof of Proposition \ref{hyper2}.

\medskip

We now prove Theorem \ref{hyper}.

\medskip

{\bf Proof of Theorem \ref{hyper}.} We set $$E = \{p \in \mathbb{H}^n | F ~\text{is ~not ~CR~transversal ~to } \mathbb{H}^N_1 \text{~at~}p \}. $$ The set $E$ is a real analytic subvariety near 0 in $\mathbb{H}^n$. If $E$ contains an open neighborhood of 0, then it follows from Theorem 1.1 in \cite{BER2} (cf. \cite{BH} Lemma 4.1) that $F(U) \subset \mathbb{H}^N_1$, which contradicts the assumption. Thus $E$ is a proper real analytic subvariety of $\mathbb{H}^n$. Consequently, there exists $p \in \mathbb{H}^n$ near 0 such that $F$ is CR transversal at $p$. We recall from \cite{BEH2} that $F$ is CR transversal to $\mathbb{H}^N_1$ at $p \in M$ if and only if $\frac{\partial g_p}{\partial w}(0) \not=0$. Here $F_p=(f_p, \varphi_p, g_p)$ is obtained from $F$ by moving to the point $p$ (refer the detailed definition of $F_p$ in \cite{BH}). Thus replacing $F$ by $F_p$ if necessary, we can assume $p=0$. Then the result follows easily from Proposition \ref{hyper2}.

\medskip

Recall the well-known fact that the Cayley transformation biholomorphically maps $\mathbb{H}^N_1$ into $\partial\mathbb{B}^N_1$ minus a proper subvariety. Then Theorem
\ref{propergene1} is a consequence of Theorem \ref{hyper}.



\begin{thebibliography}{299}


\bibitem [Al] {Al}  Alexander, H.: {\em Proper holomorphic mappings in $\mathbb{C}^n$}, Indiana Univ. Math. J. 26 (1977), 137-146.
\bibitem [B] {B} Bell, S.: {\em Algebraic mappings of circular domains in $\mathbb{C}^n$}, in Several Complex Variables, Math. Notes Vol. 38, Princeton University Press, 1993, 126-135. MR 94a: 32040.
\bibitem [BEH1] {BEH1} Baouendi, M. S., Ebenfelt, P. and Huang, X.: {\em Super-rigidity for CR embeddings
of real hypersurfaces into hyperquadrics}, Adv. Math. 219 (2008), no.5, 1427-1445.


\bibitem [BEH2] {BEH2}  Baouendi, M. S., Ebenfelt, P. and Huang, X.: {\em Holomorphic mappings between hyperquadrics with small signature difference}, Amer. J. Math. 133 (2011), no. 6. 1633-1661.


\bibitem [BER1] {BER1} Baouendi, M. S., Ebenfelt, P. and Rothschild, L: {Real submanifolds in complex space ann their mappings}, Princeton University Press, Princeton, (1999).

\bibitem [BER2] {BER2}  Baouendi, M. S., Ebenfelt, P. and Rothschild, L.: {\em Transversality of holomorphic mappings between real hypersurfaces
in different dimensions}, Comm. Anal. Geom. 15 (2007), no.3. 589-611.


\bibitem [BH] {BH} Baouendi, M. S. and Huang, X.: {\em Super-rigidity for holomorphic mappings between
hyperqadrics with positive signature}, J. Differential Geom. 69 (2005), no. 2, 379-398.

\bibitem [BHR] {BHR} Baouendi, M. S., Huang, X. and Rothschild, L. : {\em Regularity of CR mappings between algebraic hypersurfaces}, Invent. Math. 125 (1996), no. 1, 13-36.

\bibitem[B]{B} Bochner, S.: {\em Curvature in Hermitian metric}, Bull. Amer. Math. Soc. 53, (1947). 179-195.

\bibitem[C]{C}Calabi, E.: {\em Isometric imbedding of complex manifolds}, Ann.
of Math. (2) 58, (1953). 1--23, MR0057000, Zbl 0051.13103.


\bibitem[CM]{CM} Chan, S. T. and Mok, N.: {\em Holomorphic isometric embeddings of complex hyperbolic space forms into irreducible bounded symmetric domains arise from linear sections of minimal embeddings of their compact duals}, preprint, available at http://hkumath.hku.hk/~nmok/


\bibitem[CS]{CS} Cima, J. A. and Suffridge, T. J.: {\em Boundary behavior of rational proper maps}, Duke Math. J.  60  (1990),  no. 1, 135--138, MR1047119, Zbl 0694.32016.

\bibitem[CU]{CU}Clozel, L. and Ullmo, E.: {\em Correspondances modulaires et
mesures invariantes}, J. Reine Angew. Math. 558 (2003), 47--83, MR1979182, Zbl 1042.11027.



\bibitem[D1]{D1} D'Angelo, J. P.: {\em Proper holomorphic maps between balls of different dimensions}, Michigan Math. J. 35 (1988), no. 1, 83-90.

\bibitem[D2]{D2} D'Angelo, J. P.: {Several complex variables and the geometry of real hypersurfaces}, Studies in Advanced Mathematics. CRC Press, Boca Raton, FL, 1993. xiv+272 pp. ISBN: 0-8493-8272-6

\bibitem[DL] {DL} D'Angelo, J. P. and Lebl, J.: {\em Hermitian symmetric polynomials and CR complexity}, J. Geom. Anal. 21 (2011) 599-619.

\bibitem[DiL] {DiL} Di Scala, A. and Loi, A.: {\em K\"ahler manifolds and their relatives}, Ann. Sc. Norm. Super. Pisa Cl. Sci. 5 (9) (2010), no. 3, 495-501.

\bibitem[Do] {Do} Dor, A.: {\em Proper holomorphic maps between balls in one co-dimension}, Ark. Mat. 28 (1990), no. 1, 49-100.







\bibitem[Eb]{Eb} Ebenfelt, P.: {\em Partial rigidity of degenerate CR embeddings into spheres}, Adv. Math. 239, 72-96 (2013).


\bibitem [EHZ1] {EHZ1} Ebenfelt, P., Huang, X., and Zaitsev, D.: {\em Rigidity of CR-immersions into spheres}, Comm. Anal. Geom. 12 (2004), no. 3, 631-670.

\bibitem [EHZ2] {EHZ2} Ebenfelt, P., Huang, X., and Zaitsev, D.: {\em The equivalence problem and
rigidity for hypersurfaces embedded into hyperquadrics}, Amer. J. Math. 127 (2005), no. 1, 169-191.


\bibitem[ES]{ES} Ebenfelt, P. and Shroff, R.: {\em Partial rigidity of CR embeddings of real hypersurfaces into hyperquadrics with small signature difference}, Comm. Anal. Geom. 23 (2015), no. 1, 159-190.

\bibitem[FHX]{FHX} Fang, H., Huang, X. and Xiao, M.: {\em Volume-preserving maps between Hermitian symmetric spaces of compact type},  arXiv:1602.01900.

\bibitem[Fa]{Fa} Faran, J.: {\em On the linearity of proper maps between balls in the lower codimensional case}, J.
Differential Geom. 24, (1986), 15-17.


\bibitem[Fo1]{Fo1} Forstneric, F.: {\em Embedding strictly pseudoconvex domains into balls}, Trans. Amer. Math. Soc. 295 (1986), no. 1, 347-368.

\bibitem[Fo2]{Fo2} Forstneric, F.: {\em Extending proper holomorphic mappings of positive codimension}, Invent. Math. \textbf{95}, 31-62 (1989), MR0969413, Zbl 0633.32017.

\bibitem[Gl]{Gl} Globevnik, J.: {\em Boundary interpolation by proper holomorphic maps}, Math. Z. 194 (1987), no. 3,
365-373.

\bibitem[HS]{HS} Hakim, M. and Sibony, N.: {\em Fonctions holomorphes born\'ees et limites tangentielles}, Duke Math. J. 50 (1983), no. 1, 133-141.

\bibitem[HN] {HN} Henkin, G. and Novikov, R.: {\em Proper mappings of classical domains},
in Linear and Complex Analysis Problem Book, Lecture Note in Math. Vol. 1043, Springer, Berlin, 1984, 625-627.

\bibitem[H1]{H1} Hua, L. K.: {\em On the theory of Fuchsian functions of several variables}, Ann. of Math. (2) 47, (1946). 167-191.

\bibitem[H2]{H2} Hua, L. K.: {Harmonic analysis of functions of several complex variables in the classical domains}, (translated from the Russian, which was a translation of the Chinese original) Translations of Mathematical Monographs, 6. American Mathematical Society, Providence, R.I., 1979. iv+186 pp. ISBN: 0-8218-1556-3


\bibitem[Hu1] {Hu1} Huang, X.: On a linearity problem for proper holomorphic maps between
balls in complex spaces of different dimensions, J. Differential Geom. 51 (1999), no. 1, 13-33.


\bibitem[Hu2]{Hu2} Huang, X.: {\em On a semi-rigidity property for holomorphic
maps}, Asian J. Math. \textbf{7}(2003), no. 4, 463--492. (A special
issue dedicated to Y. T. Siu on the ocassion of his 60th birthday.) MR2074886, Zbl 1056.32014.

\bibitem[HJ]{HJ1} Huang, X. and Ji, S.:  {\em Mapping $B^n$ into $B^{2n-1}$}, Invent. Math. 145(2), 219-250 (2001)


\bibitem[HJY]{HJY} Huang, X., Ji, S. and Yin, W.: {\em On the third gap for proper holomorphic maps between balls}, Math. Ann. 358 (2014), no. 1-2, 115-142.

\bibitem[HY1]{HY1} Huang, X. and Yuan, Y.: {\em Holomorphic isometry from a K\"ahler manifold into a product of complex projective manifolds}, Geom. Funct. Anal. 24 (2014), no. 3, 854-886.

\bibitem[HY2]{HY2} Huang, X. and Yuan, Y.: {\em Submanifolds of Hermitian symmetric spaces}, Analysis and Geometry, 197-206, Springer Proc. Math. Stat., 127, Springer, 2015.


\bibitem[KZ1]{KZ1} Kim, S.-Y. and Zaitsev, D.: {\em Rigidity of CR maps between Shilov boundaries of bounded symmetric domains}, Invent. Math. 193 (2013), no. 2, 409-437.

\bibitem[KZ2]{KZ2} Kim, S.-Y. and Zaitsev, D.: {\em Rigidity of proper holomorphic maps between bounded symmetric domains}, Math. Ann. 362 (2015), no. 1-2, 639-677.

\bibitem [KLX] {KLX} Kossovskiy, I., Lamel, B. and Xiao, M.: {\em Regularity of CR mappings into Levi-degenerate hypersurfaces}, preprint.

\bibitem[L]{Lo} L\o w, E.: {\em Embeddings and proper holomorphic maps of strictly pseudoconvex domains into polydiscs and
balls}, Math. Z. 190 (1985), no. 3, 401-410.


\bibitem[M1]{M1} Mok, N.: {Metric rigidity theorems on Hermitian locally symmetric manifolds}, Series in Pure Mathematics. 6. World Scientific Publishing Co., Inc., Teaneck, NJ, 1989. xiv+278 pp.

\bibitem[M2]{M2} Mok, N.: {\em Local holomorphic isometric embeddings arising
from correspondences in the rank-1 case}, Contemporary trends in
algebraic geometry and algebraic topology (Tianjin, 2000), 155--165,
Nankai Tracts Math., 5, World Sci. Publ., River Edge, NJ, 2002, MR1945359, Zbl 1083.32019.


\bibitem [M3] {M08} Mok, N.: {\em Nonexistence of proper holomorphic maps between certain classical bounded symmetric domains}, Chin. Ann. Math. Ser. B 29 (2008), no. 2, 135-146.


\bibitem[M4]{M3} Mok, N.: {\em Geometry of holomorphic isometries and related maps between bounded domains}, Geometry and analysis. No. 2, 225-270, Adv. Lect. Math. (ALM), 18, Int. Press, Somerville, MA, 2011.

\bibitem[M5]{M4} Mok, N.: {\em Extension of germs of holomorphic isometries
up to normalizing constants with respect to the Bergman metric}, J. Eur. Math. Soc. 14 (2012), no. 5, 1617-1656.


\bibitem[M6]{M5} Mok, N.: {\em Holomorphic isometries of the complex unit ball into irreducible bounded symmetric domains}, preprint, available at http://hkumath.hku.hk/~nmok/





\bibitem[MN]{MN} Mok, N. and Ng, S.: {\em Germs of measure-preserving holomorphic maps from bounded symmetric domains to their Cartesian products}, J. Reine Angew. Math.  669 (2012), 47-73.


\bibitem[Ng1]{Ng1} Ng, S.: {\em On holomorphic isometric embeddings of the unit disk into polydisks}, Proc. Amer. Math. Soc. 138 (2010), 2907-2922, MR2644903, Zbl 1207.32018.

\bibitem[Ng2]{Ng2} Ng, S.: {\em On holomorphic isometric embeddings of the unit n-ball into products of two unit m-balls}, Math. Z. 268 (2011) no. 1-2, 347-354, MR2805439, Zbl 1225.53050.

\bibitem[Ng3]{Ng3} Ng, S.: {\em Proper holomorphic mappings on flag domains of SU(p,q) type on projective spaces}, Michigan Math. J. 62 (2013), no. 4, 769-777.

\bibitem[Ng4]{Ng4} Ng, S.: {\em Holomorphic Double Fibration and the Mapping Problems of Classical Domains},  Int. Math. Res. Not. IMRN 2015, no. 2, 291-324.

\bibitem[P]{P} Poincar\'e, H. : {\em Les fonctions analytiques de deux variables et la repr\'esentation conforme}, Rend. Circ. Mat. Palermo (2) 23 (1907), 185-220.




\bibitem[St]{St} Stensones, B.: {\em Proper maps which are Lipschitz up to the boundary}, J. Geom. Anal. 6 (1996), no. 2, 317-339.


\bibitem[Ts]{Ts} Tsai, I-H.: {\em Rigidity of proper holomorphic maps between symmetric domains}, J. Differential Geom. 37 (1993), 123-160.

\bibitem[Tu1]{Tu1} Tu, Z.: {\em Rigidity of proper holomorphic mappings between equidimensional bounded symmetric domains}, Proc. Amer. Math. Soc. 130 (2002), no. 4, 1035-1042.

\bibitem[Tu2]{Tu2} Tu, Z.: {\em Rigidity of proper holomorphic mappings between nonequidimensional bounded symmetric domains}, Math. Z. 240 (2002), no. 1, 13-35.

\bibitem[TH] {TH} Tumanov, A. and Henkin, G.: {\em Local characterization of holomorphic
automorphisms of classical domains}, Dokl. Akad. Nauk SSSR 267 (1982), 796-799. (Russian)
MR 85b:32048

\bibitem[Um]{Um} Umehara, M.: {\em Diastasis and real analytic functions on complex manifolds,} J. Math. Soc. Japan. 40 (1988), no. 3, 519-539.


\bibitem[W]{W} Webster, S. M.: {\em The rigidity of C-R hypersurfaces in a sphere}, Indiana Univ. Math. J. 28 (1979), 405-416.


\bibitem[XY]{XY} Xiao, M. and Yuan, Y.: {\em Holomorphic isometries from the unit ball into the irreducible bounded symmetric domain}, preprint.

\bibitem[Yu]{Yu} Yuan, Y.: {\em On local holomorphic maps preserving invariant (p,p)-forms between bounded symmetric domains}, arXiv:1503.00585.

\bibitem[YZ]{YZ} Yuan, Y. and Zhang, Y.: {\em Rigidity for local holomorphic isometric embeddings from $B^n$ into $B^{N_1} \times ... \times B^{N_m}$ up to conformal factors},  J. Differential Geom. 90 (2012), no. 2, 329-349.
\bigskip

\noindent Ming Xiao, mingxiao@illinois.edu, Department of Mathematics, University of Illinois Urbana-Champaign, IL 61801, USA.

 \noindent Yuan Yuan, yyuan05@syr.edu, Department of Mathematics, Syracuse University, NY 13244, USA.



\end{thebibliography}
\end{document}